\documentclass[11pt,a4paper]{article}
\usepackage{a4wide,amsfonts,amsmath,latexsym,amsthm,amssymb,euscript,eufrak,graphicx,units,mathrsfs,setspace,stmaryrd}
\usepackage{cancel}
\usepackage{lipsum}  

\usepackage[pdftex]{color}
\usepackage[dvipsnames]{xcolor}
\definecolor{darkblue}{rgb}{0.3,0.3,0.7}

\usepackage{float}
\newfloat{figure}{H}{lof}
\floatname{figure}{\figurename}

\usepackage[french,english]{babel}
\usepackage[T1]{fontenc}
\usepackage[utf8]{inputenc}

\DeclareMathAlphabet{\eufrak}{U}{}{}{}  
\SetMathAlphabet\eufrak{normal}{U}{euf}{m}{n}
\SetMathAlphabet\eufrak{bold}{U}{euf}{b}{n}

\newtheorem{prop}{Proposition}[section]

\newtheorem{theorem}[prop]{Theorem}
\newtheorem{lemma}[prop]{Lemma}
\newtheorem{corollary}[prop]{Corollary}

\theoremstyle{definition}

\newtheorem{remark}[prop]{Remark}

\newtheorem{definition}[prop]{Definition}

\newtheorem{notation}[prop]{Notation}
\newtheorem{convention}[prop]{Convention}

\numberwithin{equation}{section}

\def\E{\mathbb{E}}
\def\P{\mathbb{P}}
\def\real{\mathbb{R}}
\def\t{{t}}
\def\T{{T}}

\def\F{\mathcal{F}}
\def\1{\textbf{1}}
\def\ind#1{\textbf{1}_{\left\{#1\right\}}}

\def\X{\mathbb{X}}

\def\bZ{\boldsymbol{\zeta}}

\newcommand{\be}{\begin{equation}}
\newcommand{\ee}{\end{equation}}
\newcommand{\bde}{\begin{displaymath}}
\newcommand{\ede}{\end{displaymath}}
\newcommand{\beq}{\begin{eqnarray*}}
\newcommand{\eeq}{\end{eqnarray*}}
\newcommand{\beqa}{\begin{eqnarray}}
\newcommand{\eeqa}{\end{eqnarray}}
\newcommand{\bel }{\left\{\begin{array}{ll}}
\newcommand{\eel}{\cr \end{array} \right.}

\newcommand{\bex}{\begin{ex} \rm }
\newcommand{\eex}{\end{ex}}

\def\R{\mathbb R}
\def\N{\mathbb N}
\def\E{\mathbb E}
\def\F{{\cal F}}

\def\P{\mathbb P}

\def\Z{\mathcal Z}

\def\bP{\boldsymbol{\varphi}}
\def\bPE{\boldsymbol{\overline{\varphi}}}

\def\bN{\bold{N}}

\def\bM{\boldsymbol{\mu}}
\def\bO{\boldsymbol{\omega}}
\def\bX{\bold{X}}
\def\bI{\bold{Id}}

\def\br{\boldsymbol{\mathfrak{m}}}
\def\bx{\bold{x}}
\def\by{\bold{y}}

\def\bD{\bold{D}}
\def\bpI{\boldsymbol{\mathcal I}}
\def\T{\mathcal T}
\def\bnu{\boldsymbol{\nu}}

\DeclareMathOperator{\TbPs}{\overline{\boldsymbol{\widetilde{\Psi}}}}
\DeclareMathOperator{\bPs}{\overline{\boldsymbol{\Psi}}}

\def\myZ#1{\boldsymbol{Z}_{#1}^{\, i}}

\def\bZeta{\boldsymbol{\zeta}}
\def\bZetaE{\boldsymbol{\overline{\zeta}}}
\def\bLamb{\boldsymbol{\lambda}}

\def\myVarLamb#1{^{(#1)}\hspace{-0.05cm}\boldsymbol{\lambda}}
\def\myVarZ#1{^{(#1)}\hspace{-0.05cm}\mathscr{Z}}

\def\cal#1{\mathcal{#1}}

\def\b{\textcolor{darkblue}}
\def\r{\textcolor{BrickRed}}
\def\gr{\textcolor{Green}}

\definecolor{ying}{rgb}{0.8, 0.0, 0.04}

\DeclareSymbolFontAlphabet{\mathrsfs}{rsfs}

\author{Caroline Hillairet\footnote{ENSAE  IP Paris, CREST UMR 9194,
5  avenue Henry Le Chatelier
91120 Palaiseau, France.  Email: \texttt{caroline.hillairet@ensae.fr}} \and Thomas Peyrat\footnote{{EXIOM Partners and CREST UMR 9194,
5  avenue Henry Le Chatelier
91120 Palaiseau, France.}  \; Email: \texttt{thomas.peyrat@ensae.fr}} \and Anthony R\'eveillac\footnote{INSA de Toulouse, IMT UMR CNRS 5219, Universit\'e de Toulouse, 135 avenue de Rangueil 31077 Toulouse Cedex 4 France. \; Email: \texttt{anthony.reveillac@insa-toulouse.fr}} }

\title{Multivariate Self-Exciting Processes with Dependencies 
\\
}

\begin{document}

\maketitle

\allowdisplaybreaks

\begin{abstract}
\noindent
This paper introduces the class of multidimensional self-exciting processes with dependencies (MSPD), which is a unifying writing for a large class of processes: counting, loss, intensity, and also shifted processes. The  framework takes into account  dynamic dependencies between
the frequency and the severity components of the risk, and therefore induces theoretical challenges in the computations of risk valuations.
We present  a general method for calculating  different quantities related to these MSPDs, which combines the Poisson imbedding, the pseudo-chaotic expansion and Malliavin calculus. The methodology is illustrated for the  computation of  explicit general correlation formula.
\end{abstract}


\section{Introduction}
\label{section:intro}


Risk analysis for  credit or actuarial portfolios  is usually based on the study of properties
of the  so-called cumulative loss process $(L_t)$ over a period of time  $[0, T]$ where $T > 0$
denotes the maturity of a contract or the time-horizon:
$$
L_T = \sum_{i=1}^{N_T}Y_i, \quad \quad T \geq 0.
$$
 $(N_t)_{t\geq 0}$  is a counting process that models the occurence of the claims (as the defaults for a credit portfolio, or the losses for an insurance portfolio), while the random variables $(Y_i)_i $ model the claims amounts. In the classic Cramer Lundberg model, $(N_t)_{t\geq 0}$ is a Poisson process  and $(Y_i)$ is a family of positive $iid$ random variables independent of $(N_t)_{t\geq 0}$.  
However, these assumptions prevent this model from being used on certain risks where contagion or dependency phenomena have been observed, as e.g. in credit risk (see \cite{errais2010affine,schmidt2017shot,embrechts2011multivariate} among others) or cyber risk (see \cite{baldwin2017contagion,bessy2021multivariate}). In contrast, Hawkes processes, introduced in  the 1970s by  A. Hawkes,  turned out to be relevant for capturing excitation phenomena, and are now used in many fields (earthquake modeling, neuroscience, Limit order books in finance, credit risk, cyber risk,   etc...). A Hawkes process $(H_t)_{t\geq0}$ is characterized by a stochastic intensity process $(\lambda_t)_{t\geq0}$ which is  a deterministic functional of the past trajectory of the counting process itself, explicitly given as follows 
$$
\lambda_t = \mu + \int_{0}^{t}\Phi(t-s)dH_s, \quad \quad t \geq 0
$$
where $\mu >0$ is the baseline intensity and $\Phi: \R_+\rightarrow\R_+$ is the self-excitation kernel. In this standard linear Hawkes model, all claims have the same contagiousness pattern, which may not be usually the case. In particular, it would be interesting to modulate this contagion effect by the size of the claims, which seems to be quite natural for example in credit risk (a default with a large loss given default will have a higher probability to generate other cascading defaults) and also in cyber or health insurance.
This motivates us to propose in
this paper an extension of a Hawkes process in which the intensity (more precisely the excitation kernel $\Phi$)  is affected  by the claims sizes, and thus introducing  dependencies between the severity components (the $(Y_i)$) and the frequency components (the counting process). 
This process can be seen as a system of weakly SDEs with respect to a Poisson measure $N(dt,d\theta,dy)$ on $(\R^+)^3$.  This representation is known as Poisson imbedding, with an   extra dimension (here $dy$) to take into account the impact of the claims sizes $Y_i$ :
$$
\left\lbrace
\begin{array}{l}
\lambda_T:=\mu + \int_{(0,T)\times \R_+^2} \varphi(T-t,y) \ind{\theta \leq \lambda_t} N(dt,d\theta,dy),\\ \\
H_T = \int_{(0,T]\times \R^2_+} \ind{\theta \leq \lambda_t} N(dt,d\theta,dy) ,\\ \\
	L_T = \int_{(0,T]\times \R_+^2} y \,  \ind{\theta \leq \lambda_t} N(dt,d\theta,dy) \hspace{1cm} T\geq 0 .\\	
\end{array}
\right.
$$
Such model has been first introduced in Khabou \cite{khabou2022approximation} {in the case $\bP(T-t,y) = \Phi(T-t)b(y)$}. We study here a multivariate version  to better capture the different components of  the risk and their cross dependencies.
Thus, the purpose of this article is to present  theoretical results on a general  $d$-dimensional stochastic process called {\it Multivariate Self-Exciting Processes with Dependencies  (MSPD)} $\boldsymbol{Z}$, which  is a unifying writing for a large class of processes:  counting, loss, intensity, and also shifted processes that will be defined hereafter. It is useful for several applications in finance and insurance, taking into account  dynamic dependencies between the frequency and the severity components of the risk. 
A MSPD $\boldsymbol{Z}$  is defined through the imbedding procedure: for each component $\myZ{T}$, $i=1,\cdots, d$
\begin{equation}\label{equation:Z_1}
	\left\lbrace
	\begin{array}{l}
		\myZ{T}= \int_{\{1,\ldots, d\}\times(0,T]\times \R_+^2} \bZeta^{i,k}(T-t,y) \ind{\theta \leq \lambda^k_t} \bN(dk,dt,d\theta,dy) \\\hspace{11cm} T\geq 0 \\
		\bLamb^i_T:= \bM^i(t) + \int_{\{1,\ldots, d\}\times(0,T]\times \R_+^2} \bP^{i,k}(T-t,y) \ind{\theta \leq \lambda^k_t} \bN(dk,dt,d\theta,dy)\\
	\end{array}
	\right.
\end{equation}
where $\bZeta$ and $\bP$ are two $d\times d$  matrices : for $1 \leq i,k  \leq d$,  $\bP^{i,k} : [0,T]\times\R_+ \to \R_+$ (impact of the $k^{th}$-dimension on the $i^{th}$-dimension)   and $\bM^i  : [0,T] \to \R_+$ are deterministic  baseline intensities.  A formal definition of this process will be given in Definition \ref{def:MSPD}, and if $\bZeta^{i,k}(t,y)=y \ind{i=k}$ (resp. $\bZeta^{i,k}(t,y)=\ind{i=k}$), the process $\boldsymbol{Z}$ coincides with the Loss process $L$ (resp. the counting process $H$).  \\

Although a Multivariate Self-Exciting Processes with Dependencies  (MSPD) models more accurately the risk, the loss of the independency assumptions, compared to the tractable Cramer Lundberg model,  induces theoretical challenges in the computations of risk valuations. Therefore our objective is  to present a general method for calculating different quantities useful in risk assessment.  The methodology relies on the Poisson imbedding and Malliavin calculus, in the spirit of \cite{Hillairet_Reveillac_Rosenbaum}, which enables us to transform the computation of some expectations  as follows (known as Mecke formula): 
\begin{equation}
	\E\left[F \int_{\bX} h \, d\bN \right] = \int_{\bX} h(\bx) \E\left[F\circ \varepsilon_{\bx}^{+}\right] \br(d\bx).
	\end{equation}
Here $\bx= (k, t,\theta,y)  \in \bX$ and $ \br$ is the intensity  measure of $\bN$,  and   the notation $F\circ \varepsilon_{\bx}^{+}$
 denotes the functional on the Poisson space where a deterministic
jump is added to the paths of $\bN$ at time  $t$. This expression turns out to be particularly
interesting from an actuarial point of view since adding a jump at some time $t$ corresponds
to computing a  stress scenario by adding artificially a claim at time $t$. Such processes $F\circ \varepsilon_{\bx}^{+}$ will be called {\it shifted processes} in the sequel. They are related to the 
vertical derivatives in  the functional It\^o calculus of Cont and Fourni\'e  in \cite{cont2013functional}.  In particular, our  methodology uses the pseudo-chaotic expansion for a counting process, which has been recently developed in \cite{Caroline_Anthony_chaotic}.


The paper is organized as follows.  Section 2  provides a description of the Poisson space  and introduces the MSPD processes.  Section 3  recalls some elements of Malliavin calculus and the Mecke formula.  Section 4 develops the pseudo-chaotic expansion for MSPD process. As an illustration, Section 5 demonstrates how the methodology can be applied to compute the expectation and correlation of MSPD.  Section 6 concludes the paper by outlining  potential directions for future work and development.

\section{Multivariate Self-Exciting Processes with Dependencies}
This section  introduces  the theoretical framework of the Poisson measure and  defines the main mathematical objects of the paper. 
\subsection{Configuration space and the Poisson measure}
We  first introduce the configuration space and the Poisson measure, taking the main elements from \cite{Last2016} and  \cite[Chapter 6]{Privault_2009}.\\
We set $\mathbb{N}^*:=\mathbb{N} \setminus \{0\}$ the set of positive integers. 
We fix $\X:=\R_+^3$ equipped with the Borelian $\sigma$-field and we make use of the notation $x:=(t,\theta,y) \in \X$.
In this paper we fix $\bX:=\{1,\dots,d\}\times\R_+^3$ equipped with the Borelian $\sigma$-field $\mathcal X$. Thus, elements of $\bX$ will be written as $\bx:=(k,x); k \in [\![1,d]\!]; x=(t,\theta,y)\in \X$.
 $\br$ the $\sigma$-finite measure on $(\bX,\mathcal X)$ is defined as follows: for $f: \bX \rightarrow \R_+$
a measurable and bounded function
\begin{align*}
	\int_{\bX} f(\bx)\br(d\bx)&:=\sum_{k=1}^d \int_{\X}f(k,t,\theta,y)\bnu^{k}(dy)d\theta dt
\end{align*}
where for all $1\leq k \leq d$, $\bnu^{k}$  are independent probability measures on $\R_+$ and $dx$ denotes the Lebesgue measure on $\X$.\\

\noindent
We define $\Omega$ the space of configurations on $\bX$ as 
\begin{equation}
	\label{eq:Omega}
	\Omega:=\left\{\bO=\sum_{j=1}^{n} \delta_{\bx_j}, \; \bx_j \in \bX, \;j=1,\ldots,n, \; n\in \mathbb{N}\cup\{+\infty\} \right\}.
\end{equation}

\noindent
Let $\P$ the Poisson measure on $ \Omega$ under which 
the canonical evaluation $\bN$ defines a Poisson random measure with intensity measure $\br$. To be more precise given any element $\bold{A}$ in $ \mathcal X$ with $\br(\bold{A})>0$, the random variable
$$(\bN(\bO))(\bold{A}):= \bO(\bold{A})$$
is a Poisson random variable with intensity $\rho(\bold{A})$. We denote by $\mathcal{F}^{\bN}:=\sigma\{\bN(\bold{A}); \bold{A} \in  \mathcal X\}$ the $\sigma$-field generated by $\bN$. We set $\mathbb{F^N}:=\left(\F_t^N\right)_{t\geq0}$ the natural history of $\bN$ given by $\F_t^N:=\sigma\{\bN(\boldsymbol{\T} \times \boldsymbol{A}), \boldsymbol{\T}\subset \mathcal{B}([0,t]),\boldsymbol{A} \in  \mathcal{B}(\{1,\dots,d\}\times\R_+^2)\}$ where $\mathcal{B}$ denotes the  Borelian $\sigma$-field on the corresponding set.
In the case of  processes defined up to a fixed horizon $T$, we use the restrictions $\X_T:=[0,T]\times \R_+^2$ and $\bX_T:= \{1,\dots,d\}\times\X_T$.

\subsection{Kernels and convolutions}
 This section introduces  some notations, definitions and hypothesis on the matrices  $\bZeta$ and $\bP$  involved in the Poisson imbedding procedure (\ref{equation:Z_1}) for the   construction of a  Multivariate Self-Exciting Process with Dependencies (MSPD).

\begin{convention}[Matrices and scalars]
	In order to distinguish between scalars and matrices, we adopt the convention that matrices will be written in bold. If $\boldsymbol{M}$ represents a matrix, $\boldsymbol{M}^{i,k}$ corresponds to the element of $\boldsymbol{M}$ positioned in the \textit{i-th} row and \textit{k-th} column and $\boldsymbol{M}^{i,.}$ corresponds to \textit{i-th} row of the matrix $\boldsymbol{M}$. 
	If $\boldsymbol{V}$ represents a vector, then $\boldsymbol{V}^i$ will be the \textit{i-th} component of the vector. Moreover we design by $\boldsymbol{\bI_d}$ the \textit{d-}dimensional identity matrix and $\mathcal M_{d,d}^+$ denotes the set of $d\times d$ matrices with entries in $\R_+$. 
\end{convention}

\begin{definition}[\textit{d-}kernel]
	Let $d\in \N_+$ and $(\bnu^k)_{1\leq k \leq d}$ a family of probability densities on $\mathbb R^+$, then $\bZeta$ is a \textit{d-}kernel if $\bZeta$ is a $\mathcal M_{d,d}^+$-valued map such that each component $\bZeta^{i,k}(t,y): \R_+^2\to\R_+$ satisfies $\int_{\R^+} \bZeta^{i,k}(t,y)\bnu^k(dy)<+\infty$, for $(t,i,k)\in \R_+\times[\![1,d]\!]^2$. 
\end{definition}

\begin{definition}[Separable \textit{d-}kernel]\label{def:separable}
	A \textit{d-}kernel $\bZeta$ is said to be separable if $\forall(i,k)\in [\![1,d]\!]^2$ there exists two functions $\phi^{i,k}, b^{i,k}: \R_+ \to \R_+$  such that $\bZeta^{i,k}(t,y):=\phi^{i,k}(t)b^{i,k}(y)$. In this case, we denote
	$\bZeta (t,y)=  \boldsymbol{\Phi}(t) \star\boldsymbol{B}(y)$ where $\star$ is the Hadamard (that is elementwise) product. 
\end{definition}

\begin{notation}[$\boldsymbol{\nu}$-mean of a \textit{d-}kernel]
	Let  $\bZeta$ a \textit{d-}kernel associated to the measure  $\boldsymbol{\nu}=(\bnu^k)_{1\leq k \leq d}$. We denote by $\boldsymbol{\overline{\zeta}}: \R_+ \to\mathcal M_{d,d}^+$ the matrix in which each coefficient $\boldsymbol{\overline \zeta}^{i,k}$ is the mean of ${\bZeta}^{i,k}$ with respect to the measure $\bnu^k$. That is, 
\begin{equation}\label{equation:PhibE}
	\boldsymbol{\overline{\zeta}}(t):=\left(\int_0^{+\infty}\bZeta^{i,k}(t,y)\bnu^k(dy) \right)_{1\leq i,k\leq d}.
\end{equation}
\end{notation}

\begin{definition}[Non-explosive $d$-kernel]
	\label{assumption:Phi}
		A   \textit{d-}kernel  $\bZeta$ is  said to be non-explosive if  its spectral radius 	$
		\mathcal{R}(\boldsymbol{K}) < 1$ and 
		$$
		\boldsymbol{K}:=\int_0^{+\infty}\boldsymbol{\overline{\zeta}}(t)dt < + \infty .
		$$	
\end{definition}
\noindent As it is usual, the intensity of a counting process should be a predictable process. This requires to add a property on the kernel $\bP$ that defines the intensity $\bLamb$ in \eqref{equation:Z_1}.

\begin{definition}[Self-excitation \textit{d-}kernel]\label{def:self-excitation}
	Let $\bP$ a non-explosive \textit{d-}kernel.  $\bP$ is said to be a self-excitation kernel  if  in addition, for all $(i,k)\in [\![1,d]\!]^2$, $\bP^{i,k}(0,y):=0$.
 By convention, self-excitation kernel will be noted $\bP$, and will be used to define the intensity of a stochastic process.
 \end{definition}
\noindent Note that the condition $\bP^{i,k}(0,y)=0$  implies that 
\begin{eqnarray*}
&\hspace{-2.5cm}\int_{\{1,\ldots, d\}\times(0,T]\times \R_+^2} \bP^{i,k}(T-t,y) \ind{\theta \leq \lambda^k_t} \bN(dk,dt,d\theta,dy) \\
&= \int_{\{1,\ldots, d\}\times(0,T) \times \R_+^2} \bP^{i,k}(T-t,y) \ind{\theta \leq \lambda^k_t} \bN(dk,dt,d\theta,dy),
\end{eqnarray*} 
which ensures the predictability of the intensity. {In addition, we adopt the convention that for any $\tau<0$ we have $\bP(\tau,.):=0$.}

\begin{definition}[See \textit{e.g.} \cite{Bacryetal2013}]
	\label{prop:Phi_convole}
	Let $\bP$ a self-excitation \textit{d-}kernel. Let define the sequence of iterated convolutions of $\bPE$ such that
\begin{equation}\label{eq:Psi}
	\left\lbrace
	\begin{array}{l}
\bPE_0  \text{ denotes the Dirac distribution in }0, \\
 \bPE_1 :=\bPE, \quad \bPE_n:=\int_0^t \bPE(t-s) \bPE_{n-1}(s) ds, \quad t \in \real_+, \; n\in \mathbb{N}^*.\\
 \end{array}
	\right. 
\end{equation}
Since $\|\bPs\|_1 = (\bI-\boldsymbol{K})^{-1} - \bI = (\bI-\boldsymbol{K})^{-1} \boldsymbol{K}$, the mapping $\bPs$,
	\begin{equation}
		\label{eq:Psi}
		\bPs:=\sum_{n=1}^{+\infty} \bPE_n,
	\end{equation}
is well-defined as a limit in $L_1(\real_+;dt)$, and  by definition,
\begin{equation}\label{eq:magic_convol}
	\int_0^t \bPs(t-s) \bPE(s) ds =\bPs(t) - \bPE(t).
\end{equation}

\end{definition}

\noindent The intensity process $\bLamb$ is defined in (1.1) by an implicit equation. For sake of completeness, the following proposition details an iterative procedure for the  construction of this process.

\begin{prop}[Intensity process]
	Consider a vector $\bM:=(\boldsymbol{\mu}^{i})_{1\leq i \leq d}$ where each component $\boldsymbol{\mu}^{i}$ is a function from $\R^+$  to  $\R^+$ and let $\bP$ a self-excitation $d$-kernel. Then the system of SDEs 
	\begin{equation}\label{eq:intensitySDE}
		\bLamb_T^i:= \bM^i(T) + \int_{\bX_T}\bP^{i,k}(T-t,y) \ind{\theta \leq \bLamb_t^k} \bN(dk,dt,d\theta,dy); \quad T\geq 0, \quad i=1,\cdots, d
	\end{equation}
	 admits a unique $\mathbb{F^N}-$predictable solution $\bLamb$.	
\end{prop}

\begin{proof}The proof follows \cite{coutin2024normal} which we extend to the multivariate case. The proof consists in iteratively constructing $\bLamb$, starting from the deterministic baseline intensity $\boldsymbol{\mu}$ in which we add successively the excitation component each time the corresponding aggregated counting process $\mathscr{Z}$ jumps (the aggregation of jumps from each component). More precisely, for $T\geq 0$,  let us initiate the procedure by 
considering a  constant  $d$-dimensional intensity and the corresponding cumulated one-dimensional counting process :
	 $$\myVarLamb{1}_T := \bM(T) \quad
	 \mbox{ and } \quad
	\myVarZ{1}_T:=\int_{\bX_T} \ind{\theta \leq \myVarLamb{1}^k_t} \bN(dk,dt,d\theta,dy).$$	
Then the sequence $\left(\myVarLamb{n},\hspace{0.01cm}\myVarZ{n}\right)_{n\in\N^*}$ is defined  by induction as follows: for all $i=1,\cdots, d$
	 $$
	 \myVarLamb{n+1}^i_T := \bM^i(T) + \int_{\bX_T} \ind{t \leq  \tau_{n}^{\mathscr{Z}}}\bP^{i,k}(T-t,y) \ind{\theta \leq  \myVarLamb{n}_t^k} \bN(dk,dt,d\theta,dy),
	 $$
	 $$\mbox {with }
	 \tau_{n}^{\mathscr{Z}}:= \inf \left\{ \tau>0 | \myVarZ{n}_\tau = n\right\}, 
	 $$
	 $$ \mbox { and }
	 \myVarZ{n+1}_T:=\int_{\bX_T} \ind{\theta \leq \myVarLamb{n}^k_t} \bN(dk,dt,d\theta,dy).
	 $$
	 Remark that   $\myVarZ{n}=  \hspace{0.01cm}\myVarZ{n+1}$ and  $\myVarLamb{n}= \hspace{0.01cm}\myVarLamb{n+1}$ on $[0, \tau_{n}^{\mathscr{Z}}]$,
	 and  $\myVarLamb{n} \leq \hspace{0.01cm} \myVarLamb{n+1}$  $\P. a.s$, thus 
	$\bLamb:= \lim \hspace{0.01cm}\myVarLamb{n}$ is well defined and  let  
	$\mathscr{Z}_T:=\int_{\bX_T} \ind{\theta \leq \bLamb^k_t} \bN(dk,dt,d\theta,dy)= \sum_{n \geq 1} \ind{\tau_{n}^{\mathscr{Z}}\leq T}.$	
Let us prove that  the increasing sequence  $(\tau_{n}^{\mathscr{Z}})_n$ converges to $+\infty$. Indeed, for any $T>0$,  by monotone convergence 
	 $$ \P( \lim_{n \to +\infty} \tau_{n}^{\mathscr{Z}}<T) = \lim_{n \to +\infty}\P(\tau_{n}^{\mathscr{Z}}<T)=  \lim_{n \to +\infty}\P(\mathscr{Z}_T\geq n).$$
	Moreover by  Markov inequality 
	$$\P(\mathscr{Z}_T\geq n) \leq \frac{\E[\mathscr{Z}_T]}{n} = \frac{\E[ \int_0^T \sum_i \boldsymbol{\lambda}^i_s ds]}{n}
	$$
As $\bP$ is a non explosive d-kernel, $\E[ \int_0^T \sum_i \boldsymbol{\lambda}^i_s ds] < \infty$ and $\lim_{n \to +\infty}\P(\mathscr{Z}_T\geq n)=0$. Therefore   $ \P( \lim_{n \to +\infty} \tau_{n}^{\mathscr{Z}}<T) = 0$ for any $T$ and $\lim_{n \to +\infty} \tau_{n}^{\mathscr{Z}}=+\infty$ $\P a.s.$. Moreover, since the intensity is constructed iteratively and  $\myVarLamb{n}= \hspace{0.01cm}\myVarLamb{n+1}$ on $[0, \tau_{n}^{\mathscr{Z}}]$, this guarantees the unicity of $\bLamb$ when $\tau_{n}^{\mathscr{Z}}$ tends to infinity.
\end{proof}

\subsection{Definition of a Multidimensional Self-exciting Process with Dependencies (MSPD)}
We define below a generic multidimensional process with self-exciting cross dependencies, a concept that encompasses several quantities useful in finance and insurance.
\begin{definition}[Multidimensional Self-exciting Process with Dependencies : MSPD]
	\label{def:MSPD}
	Let $(\Omega,\mathcal F,\P\otimes\bnu,\mathbb F:=(\mathcal F_t)_{t\geq 0})$ be a filtered probability space, $\bZeta$ a \textit{d-}kernel and $\bP$ a self-excitation \textit{d-}kernel. Moreover, let $\bM:=(\boldsymbol{\mu}^{i})_{1\leq i \leq d}$ a family of   functions  $\boldsymbol{\mu}^i : \R^+ \rightarrow \R^+$ representing  the baseline intensity. A $(\bZ,\bM,\bP)-$MSPD  $\boldsymbol{Z}:=(\myZ{})_{1\leq i \leq d}$  is an $\mathbb R^d$-valued stochastic process $(\boldsymbol{Z}_t)_{t\geq 0}$ where for every $i \in \{1,\ldots,d\}$,
	the pair $(\myZ{},\bLamb^i)$ is solution of  the two-dimensional SDEs driven by the Poisson measure $\bN$,  
	\begin{equation}\label{eq:MSPD}
	\left\lbrace
	\begin{array}{l}
		\myZ{T} := \int_{\bX_T}\bZeta^{i,k}(T-t,y) \ind{\theta \leq \bLamb_t^k} \bN(dk,dt,d\theta,dy) \\ \hspace{12cm} T\geq 0 \\
		\bLamb_T^i:= \bM^i(T) + \int_{\bX_T}\bP^{i,k}(T-t,y) \ind{\theta \leq \bLamb_t^k} \bN(dk,dt,d\theta,dy). \\
	\end{array}
	\right. 
\end{equation}
In order to easily define an $(\bZ,\bM,\bP)-$MSPD, we always adopt the  following notation:  the first component $\bZ$ refers  to the \textit{d-}kernel, the second $\bM$ to the baseline intensity, and the third component $\bP$ to the self-excitation \textit{d-}kernel.


\end{definition}

\begin{remark}
	Let us emphasize that the process $\bLamb^i$ is  predictable, while $\myZ{}$ is not. Indeed,  when integrating with respect to  $\bN(dt)$, $\myZ{T}$ can be charged by a point $\bx_n$ having an arrival time $\t_n = T$ (assuming $\ind{\theta_n \leq \bLamb_T^{k_n}}=1$, $\myZ{T} = \myZ{T-} + \bZeta(0,y_n)$) whereas $\lambda_t^i$ can not since $\bP$ is a self-excitation \textit{d-}kernel (therefore satisfying $\bP(0,y) = 0$). 
	One standard kernel  in the literature is the exponential kernel defined by $\phi(u):=\alpha e^{-\beta u}$. This kernel can be adapted in order to fit under the framework of this paper by using $\phi(u):=\alpha e^{-\beta u}\ind{u>0}$.
\end{remark}

\begin{remark}[$\bLamb$ as an MSPD]\label{remark:lambda_MSPD}
	  (\ref{eq:MSPD})  implies that  $\bLamb_T - \bM(T)$ is also an $(\bP,\bM,\bP)-$MSPD, where  $\bP$ is a self-excitation \textit{d-}kernel.
\end{remark}

\begin{remark}[(Compound) Hawkes process as an MSPD]
As discussed  after \eqref{equation:Z_1} in the introduction, this model contains the Hawkes process (for $\bZeta(t,y)=\boldsymbol{\bI_d}$) and the compound Hawkes process (for $\bZeta(t,y)= y \boldsymbol{\bI_d}$)  as particular cases. If the excitation kernel $\bP$  does not depend on $y$,   then there is no impact of the claims sizes (severity component) on the time-arrivals of claims (frequency component).
\end{remark}

Our objective is to deploy a general methodology for the computation of different quantities in the space of random Poisson measures. This methodology is based on two essential tools, namely Mecke's formula and the pseudo-chaotic expansion \cite{Caroline_Anthony_chaotic}, which requires the introduction of some operators on this space.  
In the following section, we define elements of Malliavin's calculus in order to exploit Mecke's formula through the pseudo-chaotic expansion.

\section{Elements of Malliavin's calculus}

\subsection{Spaces and the add point operator}
We introduce some elements of Malliavin's calculus on Poisson processes. We set 
$$ L^0(\Omega):=\left\{ F:\Omega \to \real, \; \mathcal{F}^{\bN}-\textrm{ measurable}\right\},$$
$$ L^2(\Omega):=\left\{ F \in L^0(\Omega), \; \E[|F|^2] <+\infty\right\}.$$
Fix $n \in \N^*$. We set $\br^{\otimes n}$ the extension of $\br$ on $(\bX^n,\mathfrak{X}^{\otimes n})$.
Let
$$L^0(\bX^n) := \left\{f:\bX^n \to \real, \; \mathfrak X^{\otimes n}-\textrm{measurable } \right\}$$
and for $p\in\{1,2\}$,
\begin{equation}
	\label{definition:L2j}
	L^p(\bX^n) := \left\{f \in L^0(\bX^n),   \; \int_{\bX^n} |f(\bx_1,\cdots,\bx_n)|^p \br^{\otimes n}(d\bx_1 \cdots d\bx_n) <+\infty\right\}.
\end{equation}
Besides,
\begin{equation}
	\label{definition:L2js}
	L^p_s(\bX^n) := \left\{f \in L^p(\bX^n) \textrm{ and } f \textrm{ is symmetric} \right\}
\end{equation}
is the space of square-integrable symmetric mappings where we recall that $f:\bX^n \to \real$ is said to be symmetric if for any element $\sigma$ in $\mathcal S_n$ (the set of all  permutations of $\{1,\cdots,n\}$), 
$$ f(\bx_1,\ldots, \bx_n) = f(\bx_{\sigma(1)},\ldots, \bx_{\sigma(n)}), \quad \forall (\bx_1,\ldots,\bx_n) \in \bX^n.  $$

\noindent
The main ingredient we will make use of are the add-points operators on the Poisson space $\Omega$.  For any finite set $J$, we set $|J|$ its cardinal. 
\begin{definition}$[$Add-points operators$]$\label{definitin:shifts}
	Given $n\in \mathbb N^*$, and $J:=\{\bx_1,\ldots,\bx_n\} \subset \bX$ a subset of $\bX$ with $|J|=n$,
	we set the measurable mapping :
	\begin{eqnarray*}
		\varepsilon_{J}^{+,n} : \Omega & \longrightarrow & \Omega  \\
		\bO & \longmapsto   & \bO + \sum_{\bx \in J} \delta_{\bx} \ind{\bO(\{\bx\})=0}.
	\end{eqnarray*}
	Note that by definition 
	$$  \bO + \sum_{\bx \in J} \delta_{\bx} \ind{\bO(\{\bx\})=0} = \bO + \sum_{m=1}^{n} \delta_{\bx_j} \ind{\bO(\{\bx_j\})=0} $$
	that is we add the atoms $\bx_j$ to the path $\bO$ unless they already were part of it (which is the meaning of the term $\ind{\bO(\{\bx_j\})=0}$). Note that since $\br$ is assumed to be atomless, given a set $J$ as above, $\P[\bN(J)=0]=1$ hence in what follows we will simply write $\bO + \sum_{j=1}^{n} \delta_{\bx_j}$ for $\varepsilon_{\bold{x}}^{+,n} (\bO)$.  
\end{definition}

\subsection{The Malliavin derivative}

\noindent In the context of Malliavin's calculation, a tool that is often mentioned is the associated derivative.

\begin{definition}
	\label{definition:Dn}
	For $F$ in $L^2(\Omega)$, $n\in \mathbb N^*$, $(\bx_1,\ldots,\bx_n) \in \bX^n$, we set 
	\begin{equation}
		\label{eq:Dn}
		\bold{D}_{(\bx_1,\ldots,\bx_n)}^{n} F:= \sum_{J \subset \{\bx_1,\ldots,\bx_n\}} (-1)^{n-|J|} F\circ \varepsilon_{J}^{+,|J|},
	\end{equation}
	where we recall that $\emptyset \subset \bX$.
	For instance when $n=1$, we write  $\bold{D}_{\bx} F := \bold{D}_{\bx}^1 F = F(\cdot + \delta_{\bx}) - F(\cdot)$ which is the difference operator (also called add-one cost operator\footnote{see \cite[p.~5]{Last2016}}). Relation (\ref{definition:Dn}) rewrites as 
	$$ \bold{D}^n_{(\bx_1,\ldots,\bx_n)} F (\bO)= \sum_{J\subset \{1,\ldots,n\}} (-1)^{n-|J|} F\left(\omega + \sum_{j \in J} \delta_{\bx_j}\right), \quad \textrm{ for a.e. } \bO \in \Omega.$$
	Note that with this definition, for any $\bO$ in $\Omega$, the mapping 
	$$ (\bx_1,\ldots,\bx_n) \mapsto \bold{D}_{(\bx_1,\ldots,\bx_n)}^n F (\bO) $$
	belongs to $L^0_s(\bX^n)$ defined as (\ref{definition:L2js}).
\end{definition}

\begin{remark}
	\label{rk:Dderm}
	If $F$ is deterministic, then by definition $\bD^n F=0$ for any $n\geq 1$.
\end{remark}

The chaotic-chaotic expansion developed in \cite{Caroline_Anthony_chaotic} involves deterministic operators  $ \T^n$  which correspond to  iterated  Malliavin derivatives  under the specific event $\{\bN(\bX)=0\}$.

\begin{definition}
	\label{definition:patwisederivative}
	For $F\in L^0(\Omega)$, we define the deterministic operators: 
	$$ \T^0 F:= F(\emptyset),$$ 
	$$
	\left.\begin{array}{lll}
		\T^n:& L^0(\Omega) & \to L^0_s(\bX^n) \\
		&F &\mapsto \T^n F 
	\end{array}\right.
	$$
	where for any $(\bx_1,\cdots,\bx_n) \in \bX^n$, $\T^n_{(\bx_1,\cdots,\bx_n) }$ is defined by 
	$$ \T^n_{(\bx_1,\cdots,\bx_n)} F := \sum_{J\subset \{\bx_1,\cdots,\bx_n\}} (-1)^{n-|J|} F(\varpi_{J}),$$
	with $\varpi_{J} := \sum_{y_i\in J} \delta_{\by_i} \in \Omega$ for $J=\{\by_1,\ldots,\by_k\}$.
\end{definition}

\noindent
In particular, even though $F$ is a random variable, $\T^n_{(\bx_1,\cdots,\bx_n)} F$ is in $\mathbb R$ as each
term $F(\varpi_{J})$ is the evaluation of $F$ at the outcome $\varpi_{J}$. Moreover, this operator belongs to $L^1_s(\bX^n)$.

\begin{convention}[Ordered points]\label{convention:orderpoint}
	Since the operator  $\T^n_{(\bx_1,\cdots,\bx_n)} F$ is symmetric, we will assume that the points $(\bx_1,\ldots,\bx_n) \in \bX^n$ are always taken ordered with respect to the time component $t$.
\end{convention}
\subsection{Factorial measures and iterated integrals} 

\begin{prop}(Factorial measures; See \textit{e.g.} \cite[Prop 1]{Last2016}).
	\label{prop:factormeas}
	There exists a unique sequence of counting random measures $(\bN^{(n)})_{n\in \mathbb N^*}$ where for any $n$, $\bN^{(n)}$ is a counting random measure on $(\bX^n, \mathfrak X^{\otimes n})$ with
	\begin{align*}
		&\bN^{(1)}:=\bN  \quad \mbox{ and \quad  for } \quad  A \in \mathfrak X^{\otimes (n+1)},\\
		&\bN^{(n+1)}(A) \\ 
		&:= \int_{\bX^n} \left[ \int_\bX \ind{(\bx_1,\ldots,\bx_{n+1})\in A} \bN(d\bx_{n+1}) - \sum_{j=1}^n \ind{(\bx_1,\ldots,\bx_{n},\bx_j)\in A} \right] \bN^{(n)}(d\bx_1,\ldots,d\bx_n); 
	\end{align*}
\end{prop}
\noindent With this definition at hand we introduce the notion of iterated integrals. In particular for $A \in \mathfrak X$, 
$$ \bN^{(n)}(A^{\otimes n}) = \bN(A) (\bN(A)-1) \times \cdots \times (\bN(A)-n+1).$$
Note that by definition $\bN^{(n)}(A) \ind{\bN(A) < n}=0$. We now turn to the definition of iterated integrals with respect to the counting measure $\bN^{(n)}(d\bx_1,\ldots,d\bx_n)$.

\begin{definition}
	\label{definition:interated}
	Let $n \in \mathbb N^*$ and $f_n \in L^1_s(\bX^n)$.
		$ \boldsymbol{\mathcal{I}}_n(f_n)$ the $n^{th}$ iterated integral of $f_n$ with respect to  the Poisson measure $\bN$ is defined as 
		$$ \boldsymbol{\mathcal{I}}_n(f_n):= \int_{\bX^{n}}  f_n(\bx_1,\ldots,\bx_n) \; \bN^{(n)}(d\bx_1,\ldots,d\bx_n),$$
		where each of the integrals above is well-defined pathwise for $\P$-a.e. for each \textit{$\bO \in \Omega$, as a   Stieltjes integral }.

\end{definition}

\subsection{The Mecke formula}
We end this section by recalling a particular case of Mecke's formula (see \textit{e.g.} \cite[Relation (11)]{Last2016}).

\begin{lemma}[A particular case of Mecke's formula]
	\label{lemma:Mecke}
		Let $F \in L^0(\Omega)$, $n\in \mathbb N$ and $h\in L^0(\bX^n)$ 
	such that
$$\int_{\bX^n} |h(\bx_1,\ldots,\bx_n)| \E\left[|F\circ \varepsilon_{\bx_1,\ldots,\bx_n}^{+,n}|\right]  \br^{\otimes n}(d\bx_1,\ldots,d\bx_n)<+\infty.$$

Then 
	\begin{equation}\label{eq:Mecke_1}
	\E\left[F \int_{\bX^n} h d\bN^{(n)}\right] = \int_{\bX^n} h(\bx_1,\ldots,\bx_n) \E\left[F\circ \varepsilon_{\bx_1,\ldots,\bx_n}^{+,n}\right] \br^{\otimes n}(d\bx_1,\ldots,d\bx_n).
	\end{equation}
By taking $F=1$ we get
	\begin{align}\label{eq:Mecke_2}
		\E\left[\bpI_n(f_n)\right] &= \int_{\bX^n}f_n(\bx_1,\dots,\bx_n)\br^{\otimes n}(d\bx_1,\ldots,d\bx_n) \nonumber\\
		&=\sum_{k_1=1}^d \dots \sum_{k_n=1}^d \int_{([0,T]\times \R_+ \times \R_+)^n}f_n(\bx_1,\dots,\bx_n)\prod_{i=1}^n\bnu^{k_i}(dy_i)d\theta_i dt_i.
	\end{align}
\end{lemma}

\section{Pseudo-chaotic expansion for counting processes}
The second ingredient in our analysis relies on the pseudo-chaotic expansion. This section presents some essential results around this expansion, with a focus on the process $\boldsymbol{Z}$.

\subsection{Around the pseudo-chaotic expansion of $\boldsymbol{Z}$}


{The following theorem provides the  peudo-chaotic expansion for a $(\bZ,\bM,\bP)-$MSPD,  following   results from \cite{Caroline_Anthony_chaotic}.}

\begin{theorem}
	\label{th:pseudochaoticcounting}
	Let $\boldsymbol{Z}$ a $(\bZ,\bM,\bP)-$MSPD, we recall that this process is given by the SDE(\ref{eq:MSPD}):
	$$
	\left\lbrace
	\begin{array}{l}
		\myZ{T} := \int_{\bX_T}\bZeta^{i,k}(T-t,y) \ind{\theta \leq \bLamb_t^k} \bN(dk,dt,d\theta,dy) \\ \hspace{12cm} T\geq 0 \\
		\bLamb_T^i:= \bM^i(T) + \int_{\bX_T}\bP^{i,k}(T-t,y) \ind{\theta \leq \bLamb_t^k} \bN(dk,dt,d\theta,dy). \\
	\end{array}
	\right. 
	$$
	Then $\boldsymbol{Z}^{i}_T$ admits the pseudo-chaotic representation 
	\begin{equation}\label{eq:pseudo_chao}
		\myZ{T} = \lim_{M\to \infty}\sum_{n=1}^{+\infty}\frac{1}{n!}\bpI_n\left(\T^n_{(\bx_1,\cdots,\bx_n)}\myZ{T}\ind{([0,T]\times[0,M])^n}\right)
	\end{equation}
		
	with for all $(\bx_1,\ldots,\bx_n) \in \bX^n$,
	\begin{equation}
		\label{eq:ckcounting}
		\hspace*{-0.5cm}\T^n_{(\bx_1,\ldots,\bx_n)}\myZ{T} = \bZeta^{i,k_n}(T-t_n,y_n)\T^{n-1}_{(\bx_1,\dots,\bx_{n-1})}\ind{\theta_{n}\leq \bLamb^{k_n}_{t_n}}.
	\end{equation}
\end{theorem}

\begin{proof}
Equation (\ref{eq:pseudo_chao}) follows from \cite{Caroline_Anthony_chaotic} which gives the pseudo-chaotic expansion of any random linear
functional of $\bN$ restricted to a bounded domain (say $[0, T] \times [0, M]$, for $T, M > 0$) of $\R_+^2$; with a
focus on random variables of the form $H_t$ where $H$ is a counting process with bounded
intensity (we refer the reader to \cite{Caroline_Anthony_chaotic} for a complete exposition). Even though the intensity of
$\myZ{T}$ is unbounded, it is proved in \cite{Caroline_Anthony_chaotic} that marginals of Hawkes processes admit
a pseudo-chaotic expansion. Then to prove (\ref{eq:ckcounting}), recalling that $\bx = (k,t,\theta,y)$,
\begin{align*}
	\T^n_{(\bx_{1},\cdots,\bx_{n})} \myZ{T} &= \sum_{J\subset \{\bx_1,\cdots,\bx_n\}} (-1)^{n-|J|} \left(\int_{\bX_T}\bZeta^{i,k}(T-t,y) \ind{\theta \leq \bLamb_t^k} \bN(dk,dt,d\theta,dy)\right) (\varpi_{J})\\
	&= \sum_{J\subset \{\bx_1,\cdots,\bx_n\}} (-1)^{n-|J|}\sum_{\bx_m\in J} \bZeta^{i,k_m}(T-t_m,y_m)\ind{\theta_m \leq \bLamb_{t_m}^{k_m} (\varpi_{J})} \\
	&= \sum_{J\subset \{\bx_1,\cdots,\bx_n\}} \sum_{\bx_m\in J}(-1)^{n-|J|} \bZeta^{i,k_m}(T-t_m,y_m)\ind{\theta_m \leq \bLamb_{t_m}^{k_m} (\varpi_{J})}\\
	&= \sum_{m=1}^n \sum_{J\subset \{\bx_1,\cdots,\bx_n\};\bx_m\in J} (-1)^{n-|J|}\bZeta^{i,k_m}(T-t_m,y_m)\ind{\theta_m \leq \bLamb_{t_m}^{k_m} (\varpi_{J} + \delta_{\bx_m}))} \\
	&= \sum_{m=1}^n  \bZeta^{i,k_m}(T-t_m,y_m)  \sum_{J\subset \{\bx_1,\cdots,\bx_{i-1},\bx_{i+1},\cdots,\bx_n\}} (-1)^{n-1-|J|} \ind{\theta_m \leq \bLamb_{t_m}^{k_m} (\varpi_{J} + \delta_{\bx_m}))} \\
\end{align*}
Then, for $m \in \{1,\cdots,n\}$ we have that
\begin{align*}
	&\sum_{J\subset \{\bx_1,\cdots,\bx_{m-1},\bx_{m+1},\cdots,\bx_n\}} (-1)^{n-1-|J|} \ind{\theta \leq \bLamb_{t_m}^{k_m} (\varpi_{J} + \delta_{\bx_m})}\\
	&= \sum_{J^-\subset \{\bx_1,\cdots,\bx_{m-1}\}}\sum_{J^+\subset \{\bx_{m+1},\cdots,\bx_n\}}(-1)^{n-1-|J^-|-|J^+|}\ind{\theta_m \leq \bLamb_{t_m}^{k_m} (\varpi_{J^- \cup J^+} + \delta_{\bx_i})}\\
	&=\sum_{J^-\subset \{\bx_1,\cdots,\bx_{m-1}\}}\ind{\theta_m \leq \bLamb_{t_m}^{k_m} (\varpi_{J^-})} (-1)^{n-1-|J^-|}\sum_{J^+\subset \{\bx_{m+1},\cdots,\bx_n\}}(-1)^{|J^+|}
\end{align*}
where the last equality follows from the predictability of $\bLamb$.
Moreover, if $\bx_m \neq \bx_n$ then,
$$
\sum_{J^+\subset \{\bx_{m+1},\cdots,\bx_n\}}(-1)^{|J^+|} = 0.
$$
And if $\bx_m = \bx_n$ then $J^+ = \emptyset$ which implies, 
$$\sum_{J^+\subset \{\bx_{m+1},\cdots,\bx_n\}}(-1)^{|J^+|} = 1.$$
Thus, 
\begin{align*}
	\T^n_{(\bx_1,\cdots,\bx_n)}\myZ{T}&= \bZeta^{i,k_n}(T-t_n,y_n)\sum_{U\subset \{\bx_1,\cdots,\bx_{n-1}\}}(-1)^{n-1-|U|}\ind{\theta_n \leq \bLamb_{t_n}^{k_n} (\varpi_{U})}\\
	&=\bZeta^{i,k_n}(T-t_n,y_n)\T^{n-1}_{(\bx_1,\cdots,\bx_{n-1})}\ind{\theta_n \leq \bLamb_{t_n}^{k_n}}.
\end{align*}

\end{proof}

\begin{remark}\label{remark:Tnlambda}
As $\bM^i(T)$ is deterministic, $\T^{n} \bM^i(T)=0$ hence from  \eqref{eq:ckcounting} we deduce that
$$\T^n_{(\bx_1,\ldots,\bx_n)} \bLamb_T^i = \bP^{i,k_n}(T-t_n,y_n)\T^{n-1}_{(\bx_1,\dots,\bx_{n-1})}\ind{\theta_{n}\leq \bLamb^{k_n}_{t_n}}.$$
\end{remark}

\begin{remark}\label{remark:pre_compatibility}
	Consider a set  $J= (\bx_1,\cdots,\bx_n)$ of ordered points  in $\bX_T$ (see Convention \ref{convention:orderpoint}). Then  $\myZ{T}$  evaluated on $J$ satisfies,  for any $\omega\in\Omega$
\begin{eqnarray*}
	\myZ{T}\left( \sum_{\bx_m\in J} \delta_{\bx_m} \right) &=& \sum_{\bx_m\in J}\bZeta^{i,k_m}(T-t_m,y_m)\ind{\theta_{m} \leq \bLamb_{t_m}^{k_m}( \sum_{\bx_j\in J}  \delta_{\bx_j })}\\
	&\leq& \sum_{m\in J}\bZeta^{i,k_m}(T-t_m,y_m)\leq \myZ{T}\left( \omega + \sum_{\bx_m\in J} \delta_{\bx_m} \right). 
\end{eqnarray*}
Since for all $(i,k)\in \{1,\dots,d\}^2$, $\bZ^{i,k}$ is positive,	 $	\myZ{T}( \sum_{j\in J} \delta_{\bx_j})$  reaches its maximum   when every point in $J$ is accepted by the 
	thinning criteria $\ind{\theta_{m} \leq {\bLamb_{t_m}^{k_m}(\sum_{j\in J} \delta_{\bx_j})}}$. In fact, if we assume that every point in $J$ verifies the thinning criteria we have for all $\bx_m \in J$,
\begin{equation}\label{eq:Caro}
	\theta_{m}\leq \bM^{k_m}
	+  \sum_{j=1}^{m-1}\bP^{k_m,k_j}(t_m-t_j,y_j) = {\bLamb_{t_m}^{k_m}(\sum_{j\in J} \delta_{\bx_j})}\leq \bLamb_{t_m}^{k_m}(\omega + \sum_{j\in J} \delta_{\bx_j}),
	\end{equation}
	where the  sum stops at $j=m - 1$ because of  the predictability of $\lambda_T^i$ (see Definition \ref{def:self-excitation}).
\end{remark}
This brings us to introduce the following lemma which establishes a compatibility condition on a set of points $\bx_j$ such that $\T^n_{(\bx_1, \dots, \bx_n)}{\myZ{t}}\neq0$.
\begin{lemma}[Compatibility condition]Let  $\boldsymbol{Z}$ a $(\bZ,\bM,\bP)-$MSPD. Fix $T \geq 0$, let $n \in \N^*$, $(\bx_1 \dots, \bx_n)\in (\bX_T)^n$ following Convention \ref{convention:orderpoint}.
	It holds that, 
	
	$$
	\T^n_{(\bx_1, \dots, \bx_n)}{\myZ{T}}=\T^n_{(\bx_1, \dots, \bx_n)}{\myZ{T}}\ind{\theta_1 \leq \bM^{k_1}(t_1)} \prod_{m=2}^{n}\ind{\theta_m \leq \bM^{k_m}(t_m)
		+  \sum_{j=1}^{m-1}\bP^{k_m,k_j}(t_m-t_j,y_j)}.
	$$
$(\bx_1 \dots, \bx_n)$ is said to satisfy the compatibility condition if
\begin{equation}\label{eq:compatibility}
	\ind{\theta_1 \leq \bM^{k_1}} \prod_{m=2}^{n}\ind{\theta_m \leq \bM^{k_m}
		+  \sum_{j=1}^{m-1}\bP^{k_m,k_j}(t_m-t_j,y_j)} = 1.
\end{equation}
\end{lemma}
\begin{proof}
	By applying (\ref{eq:ckcounting}) from Theorem \ref{th:pseudochaoticcounting} to $\boldsymbol{Z}$ and by definition of the operator $\T$ we have that,
	$$
	\T^n_{(\bx_1, \dots, \bx_n)}{\myZ{T}}= \bZeta^{i,k_n}(T-t_n,y_n) \sum_{J\subset\{1,\dots,n-1\}}(-1)^{n-1-|J|} 
	\ind{\theta_{n} \leq \lambda_{t_n}^{k_n}(\sum_{p \in J}\delta_{\bx_p})}.
	$$
	Hence, inequality \eqref{eq:Caro} (with $m=n$) implies
	$$
	\T^n_{(\bx_1, \dots, \bx_n)}{\myZ{T}}\ind{\theta_n > \mu^{k_n} + \sum_{m = 1}^{n-1}\bP^{k_n,k_m}(t_n - t_m,y_m)}=0.
	$$
	\\
	Let $ \ell\in\{2,\dots,n\}$,
	\begin{align*}
		&\T^n_{(\bx_1, \dots, \bx_n)}{\myZ{T}}\ind{\theta_\ell > \bM^{k_\ell}(t_\ell) + \sum_{m = 1}^{\ell-1}\bP^{k_\ell,k_m}(t_k - t_m,y_m)}\\
		&= \bZeta^{i,k_n}(T-t_n,y_{n}) \ind{\theta_\ell > \bM^{k_\ell}(t_\ell) + \sum_{m = 1}^{\ell-1}\bP^{k_\ell,k_m}(t_k - t_m,y_m)} \sum_{J\subset\{1,\dots,n-1\}}(-1)^{n-1-|J|} 
		\ind{\theta_{n} \leq \bLamb_{t_n}^{k_n}(\sum_{p \in J}\delta_{\bx_p})}\\
		&=\bZeta^{i,k_n}(T-t_n,y_{n}) \ind{\theta_\ell > \bM^{k_\ell}(t_\ell) + \sum_{m = 1}^{\ell-1}\bP^{k_\ell,k_m}(t_k - t_m,y_m)} \\
		&\left(\sum_{J\subset\{1,\dots,n-1\}; \ell \in J}(-1)^{n-1-|J|} 
		\ind{\theta_{n} \leq \bLamb_{t_n}^{k_n}(\sum_{p \in J}\delta_{\bx_p})} + \sum_{J\subset\{1,\dots,n-1\}; \ell \notin J}(-1)^{n-1-|J|} 
		\ind{\theta_{n} \leq \bLamb_{t_n}^{k_n}(\sum_{p \in J}\delta_{\bx_p})}\right)\\
		&=\bZeta^{i,k_n}(T-t_n,y_{n}) \ind{\theta_\ell > \bM^{k_\ell}(t_\ell)  +\sum_{m = 1}^{\ell-1}\bP^{k_\ell,k_m}(t_k - t_m,y_m)} \\
		&(\sum_{J\subset\{1,\dots,\ell-1,\ell+1,\dots,n-1\}} \hspace*{-1cm} (-1)^{n-2-|J|} 
		\ind{\theta_{n} \leq \bLamb_{t_n}^{k_n}(\delta_{x_\ell}+\sum_{p \in J}\delta_{\bx_p})} 
		 + \sum_{J\subset\{1,\dots,\ell-1,\ell+1,\dots,n-1\}}   \hspace*{-1cm} (-1)^{n-1-|J|} 
		\ind{\theta_{n} \leq \bLamb_{t_n}^{k_n}(\sum_{p \in J}\delta_{\bx_p})}).
	\end{align*}
	Since on the domain
	$\{\theta_\ell > \bM^{k_\ell}(t_\ell) + \sum_{m = 1}^{\ell}\bP^{k_\ell,k_m}(t_\ell - t_m,y_m)\}$
	we have that,
	$$\ \bLamb_{t_n}^{k_n}(\delta_\ell + \sum_{p \in J}\delta_{\bx_p})=\bLamb_{t_n}^{k_n}(\sum_{p \in J}\delta_{\bx_p}), $$
	therefore for any  $ \ell\in\{2,\dots,n\}$
	$$
	\T^n_{(\bx_1, \dots, \bx_n)}{\myZ{T}}\ind{\theta_\ell > \bM^{k_\ell}(t_\ell) + \sum_{m = 1}^{\ell-1}\bP^{k_\ell,k_m}(t_\ell - t_m,y_m)}
	= 0 .
	$$
For $\ell=1$
	$$
	\T^n_{(\bx_1, \dots, \bx_n)}{\myZ{T}}= \ind{\theta_{1} \leq \bLamb_{t_1}^{k_1}(\delta_{\bx_1})}=\ind{\theta_{1} \leq \bM^{k_1}(t_1)}
	$$
	which concludes the proof.
\end{proof}

\subsection{Shifted processes}
Combining  Mecke's formula  (\ref{eq:Mecke_1}) 
and the pseudo-chaotic expansion (Theorem \ref{th:pseudochaoticcounting})  involves processes of the form $F\circ \varepsilon_{\bx_1,\ldots,\bx_n}^{+,n}$ that we call shifted processes. An in-depth study of these processes for $F= \myZ{}$, is of interest, as they can be interpreted as stressed scenarios.

\begin{remark}[A consequence of the compatibility condition]
We would like to enhance an important consequence of the compatibility condition when shifting the intensity of an MSPD.
Let $\boldsymbol{Z}$ a $(\bZ,\bM,\bP)-$MSPD. Let $n \in \N^*$, $(\bx_1 \dots, \bx_n)\in (\bX_T)^n$ satisfying the compatibility condition (\ref{eq:compatibility}), then for any $\omega \in \Omega$,  the shifted processes are given by
	\begin{align*}
		\myZ{T}\circ \varepsilon_{(\bx_1 \dots, \bx_n)}^{+,n}(\omega)&=\left( \int_{\bX_T}\bZeta^{i,k}(T-t,y) \ind{\theta \leq \bLamb_t^k} \bN(dk,dt,d\theta,dy)\right)(\omega + \sum_{j=1}^{n}\delta_{\bx_j})\\
		&=\left(\int_{\bX_T}\bZeta^{i,k}(T-t,y)  \ind{\theta \leq \bLamb_t^k(\omega + \sum_{j=1}^{n}\delta_{\bx_j})} \bN(dk,dt,d\theta,dy)\right)(\omega)\\
		&+\sum_{j=1}^{n}\bZeta^{i,k_j}(T-t_j,y_j) \ind{\theta_j \leq \bLamb_{t_j}^{k_j}(\omega + \sum_{j=1}^{n}\delta_{\bx_j})}.
	\end{align*}
Since $(\bx_1 \dots, \bx_n)\in (\bX_T)^n$ satisfy the compatibility condition,  Remark \ref{remark:pre_compatibility} implies 
\begin{align*}
	\myZ{T}\circ \varepsilon_{(\bx_1 \dots, \bx_n)}^{+,n}&=\sum_{j=1}^{n}\bZeta^{i,k_j}(T-t_j,y_j)+\int_{\bX_T}\bZeta^{i,k}(T-t,y)\ind{\theta \leq \bLamb_t^k\circ \varepsilon_{(\bx_1 \dots, \bx_n)}^{+,n}} \bN(dk,dt,d\theta,dy).
\end{align*}
Moreover, from Remark \ref{remark:lambda_MSPD}, $\bLamb - \bM$ can be seen as an MSPD which gives us, 
	\begin{align*}
		\bLamb_T^i\circ \varepsilon_{(\bx_1 \dots, \bx_n)}^{+,n}
		&= \bM^i(T)+ \sum_{j=1}^{n}\bP^{i,k_j}(T-t_j,y_j)  + \int_{\bX_T}\bP^{i,k}(T-t,y)  \ind{\theta \leq \bLamb_t^k\circ \varepsilon_{(\bx_1 \dots, \bx_n)}^{+,n}} \bN(dk,dt,d\theta,dy).
	\end{align*}
\end{remark}
From the results above, we can remark that a shifted process $\myZ{T}\circ \varepsilon_{(\bx_1 \dots, \bx_n)}^{+,n}$ can  be seen as  a MSPD-process whose baseline intensity is impacted by  $(\bx_1 \dots, \bx_n)$. This leads to the following definition.

\begin{definition}[Compensated shift of an MSPD]\label{def:compensated shift}
	Let $\boldsymbol{Z}$ a $(\bZ,\bM,\bP)-$MSPD and let $n \in \N^*$, $(\bx_1 \dots, \bx_n)\in (\bX_T)^n$ satisfy the compatibility condition (\ref{eq:compatibility}). Then its compensated shift $\boldsymbol{Z}\odot\varepsilon_{(\bx_1 \dots, \bx_n)}^{+,n}$ given by
	\begin{equation}
			\myZ{T}\odot\varepsilon_{(\bx_1 \dots, \bx_n)}^{+,n} := \myZ{T}\circ \varepsilon_{(\bx_1 \dots, \bx_n)}^{+,n} -  \sum_{j=1}^{n}\bZeta^{i,k_j}(T-t_j,y_j),
	\end{equation}
	is a $(\bP,\bM_{(\bx_1,\dots,\bx_n)},\bP)-$MSPD where $\bM_{(\bx_1,\dots,\bx_n)}$ is such that, 
	$$
	\bM^i_{(\bx_1,\dots,\bx_n)}(T) = \bM^i(T)+ \sum_{j=1}^{n}\bP^{i,k_j}(T-t_j,y_j).
	$$ 
	In other terms we have that the compensated shift of a $(\bZeta,\bM,\bP)-$MSPD noted $\boldsymbol{Z}_{T}\odot\varepsilon_{(\bx_1 \dots, \bx_n)}^{+,n}$, is a $(\bZeta,\bM_{(\bx_1,\dots,\bx_n)},\bP)-$MSPD.
\end{definition}

 Remark \ref{remark:lambda_MSPD} highlighted  that $\bLamb_T -\bM(T)$  is  a $(\bP,\bM,\bP)-$MSPD. This stays true in the case of the compensated shift, in fact we have
$$
\bLamb_T^i\odot\varepsilon_{(\bx_1 \dots, \bx_n)}^{+,n}:=\bLamb_T^i\circ \varepsilon_{(\bx_1 \dots, \bx_n)}^{+,n}-\bM^i(T)- \sum_{j=1}^{n}\bP^{i,k_j}(T-t_j,y_j)
$$
which is a $(\bP,\bM_{(\bx_1,\dots,\bx_n)},\bP)-$MSPD.

\section{Expectation and correlations}
This section illustrates how the previous results enable us to develop a new methodology for calculating quantities related to this MSPD process $\boldsymbol{Z}$  (in particular risk valuation). To start with, we develop the computations for the expectation and the correlations  of $\boldsymbol{Z}$. For the expectation, a well-known method consists of exploiting the fact that the expectation of the intensity satisfies a Volterra equation.
However, this method is not available for the calculation of higher-order moments. Therefore, for the computation of the covariance, we will combine the pseudo-chaotic expansion with Mecke's formula to obtain expressions involving the expectation of shifted processes that can be obtained as solution of a Volterra type equation.

\subsection{Expectation of the process $\boldsymbol{Z}$ and its shifted version}
We start with the following result that extends \cite[Theorem 2]{Bacryetal2013} and \cite[Theorem 2.4]{hillairet2023explicit}.
\begin{prop}\label{prop:expectation}
	Let $\boldsymbol{Z}$ a $(\bZ,\bM,\bP)-$MSPD. Let $p \in \N^*$, $(\bx_1 \dots, \bx_p)\in (\bX_T)^p$ satisfying the compatibility condition (\ref{eq:compatibility}). Then we have
	\begin{itemize}
		\item[(i)] Expectation of $\boldsymbol{Z}_T$ : $$\E[\boldsymbol{Z}_T] = \int_{0}^{T}\bZetaE(T-v) \left(\bM(v) + \int_{0}^{v}  \bPs(v-w)\bM(w) dw\right) dv . $$
		\item[(ii)] Expectation of $\boldsymbol{Z}_T$ shifted by $(\bx_1 \dots, \bx_p)$ :
		 \begin{align*}
		\E[\boldsymbol{Z}_T\circ \varepsilon_{(\bx_1,\dots,\bx_p)}^{+,p}]&=\E[\boldsymbol{Z}_T] +  \sum_{j=1}^{p}\bZeta^{.,k_j}(T-t_j,y_j) \\ &+\sum_{j=1}^{p}\int_{t_j}^{T}\bZetaE(T-v) \left(\bP^{.,k_j}(v-t_j,y_j) + \int_{t_j}^{v}  \bPs(v-w) \bP^{.,k_j}(w-t_j,y_j)dw\right) dv.
		 \end{align*}
	\end{itemize}
\end{prop}

\begin{proof}
	For (i), by taking the expectation of the process $\myZ{T}$ we have that, 
	$$
	\E[\myZ{T}] = \sum_{k=1}^d\int_{\bX_T} \bZetaE^{i,k}(T-t) \E[\bLamb_t^k] dt.
	$$
	Hence in order to compute $\E[\myZ{T}]$ we first need $\E[\bLamb_t^k]$, for $k=1,\cdots,d$.
	It satisfies the linear Volterra ODE 
	which reads in a matrix form as 
	$$
	\E[\bLamb_T] = \bM(T) + \int_{0}^T \bPE(T-t) \E[\bLamb_t] dt,
	$$
	whose solution is given by, 
	$$
	\E[\bLamb_T] = \bM(T)+ \int_{0}^{T} \bPs(T-w)\bM(w) dw.
	$$
	Thus,
	\begin{align*}
		\E[\boldsymbol{Z}_{T}] &= \int_{0}^T \bZetaE(T-t) \E[\bLamb_t] dt\\
		&=\int_{0}^T \bZetaE(T-t)\left(\bM(t) + \int_{0}^{t} \bPs(t-w)\bM(w)dw \right)dt.
	\end{align*}
	For (ii), we use Definition \ref{def:compensated shift} to deduce
	\begin{align*}
		\hspace{-2cm}\E[\boldsymbol{Z}_T\odot \varepsilon_{(\bx_1,\dots,\bx_p)}^{+,p}] &=\int_{0}^T \bZetaE(T-t)\left(\bM(t)+ \sum_{j=1}^{p}\bP^{.,k_j}(t-t_j,y_j)\right)dt \\
		& + \int_{0}^T \bZetaE(T-t)\left(\int_{0}^{t} \bPs(t-w)\bM(w)dw  + \int_{0}^{t} \bPs(t-w)\sum_{j=1}^{p}\bP^{.,k_j}(w-t_j,y_j)dw\right)dt\\
		&=\E[\boldsymbol{Z}_T]\\
		&+\int_{0}^T \bZetaE(T-t)\left( \sum_{j=1}^{p}\bP^{.,k_j}(t-t_j,y_j)  + \int_{0}^{t} \bPs(t-w)\sum_{j=1}^{p}\bP^{.,k_j}(w-t_j,y_j)dw\right)dt.
	\end{align*}
		
\end{proof}
\noindent In (ii) of Proposition \ref{prop:expectation}, the expectation $\E[\myZ{T}\odot \varepsilon_{(\bx_1,\dots,\bx_p)}^{+,p}]$ of a $(\bZ,\bM_{(\bx_1,\dots,\bx_p)},\bP)-$MSPD is expressed as the sum of the expectations of a $(\bZ,\bM,\bP)-$MSPD and a $(\bZ,\bM_{(\bx_1,\dots,\bx_p)} - \bM,\bP)-$MSPD. This indicates that the expectation of the processes exhibits some linearity  property with respect to the baseline intensity. This observation is made formal in the following remark.
\begin{remark}[Expectation's linearity with respect to the baseline intensity] \mbox{ }\\
		Let $\boldsymbol{Z}$ a $(\bZ,\bM,\bP)-$MSPD. Let $p \in \N^*$, and consider a family of \textit{d-}dimensional vectors $\left(\boldsymbol{\Lambda}_j\right)_{1\leq j\leq p}$ such that $\forall j \in[\![1,p]\!],\forall i \in[\![1,d]\!] \;\boldsymbol{\Lambda}_j^i:\R_+\to\R_+ $. Consider the baseline intensity given by
	$$\tilde\bM(t) = \bM(t)+ \sum_{j=1}^{p}\bold{\Lambda}_j(t),$$
	and the associated processes such that $\boldsymbol{\tilde Z}$ is a $(\bZ,\tilde \bM,\bP)-$MSPD and 
	$\boldsymbol{ \tilde \Z}$ is a $(\bZ,\sum_{j=1}^{p}\bold{\Lambda}_j(T),\bP)-$MSPD.
	We have that
	\begin{align*}
		\E[\boldsymbol{\tilde Z}_T]&=\int_{0}^{T}\bZetaE(T-v) \left(\tilde\bM(v) + \int_{0}^{v}  \bPs(v-w)\tilde\bM(w) dw\right) dv =\E[\boldsymbol{Z}_T] + \E[\boldsymbol{\tilde \Z}_T]. 
	\end{align*}
\end{remark}
\begin{remark}
Proposition \ref{prop:expectation} provides as a by product a generalization of the Wald identity which in dimension 1 reduces to  \vspace{-4mm}
$$ \E[L_T] = \E[N_T] \E[Y_1]; \quad  \text{ for } L_T:=\sum_{j=1}^{N_T} Y_j    \vspace*{-3mm}$$
where  $N$ a counting process independent from the iid random variables $(Y_j)_{j\geq 1}$. In our setting, for $1\leq i\leq d$, we prove that 
\vspace{-2mm}
 $$\E[\boldsymbol{L}^i_T] = \E[Y^i]\E[\boldsymbol{H}^i_T]   \quad  \text{ for }
  L^i_T:=\sum_{j=1}^{H_T^i} Y^i_j  \vspace*{-3mm}$$ 
 where $(Y^i_j)_{j\geq 1}$ are iid random variables with probability density $\bnu^i$, 
 but in which $\E[\boldsymbol{H}^i_T]$ is impacted by the distributions of the claims sizes $(Y^1, \cdots, Y^d)$.
\end{remark}
For computing the expectation of the intensity, we apply Proposition \ref{prop:expectation} to $(\bLamb - \bM)$ which  can be considered as a $(\bP,\bM,\bP)-$MSPD 
(see Remark \ref{remark:lambda_MSPD}).
Since the \textit{d-}kernel $\bZeta$ coincides with the self-excitation \textit{d-}kernel $\bP$, the expression for the expectation simplifies using (\ref{eq:magic_convol}).
\begin{corollary}[Intensity expectation]
	The expectation of the intensity $\bLamb^i_T$ (see (\ref{eq:MSPD}))  is given by
	$$
	\E[\bLamb_T] = \bM(T)+ \int_{0}^{T} \bPs(T-w)\bM(w) dw
	$$
	and, for $p \in \N^*$, $(\bx_1 \dots, \bx_p)\in (\bX_T)^p$,  its shifted version is given by
	\begin{align*}
		\E[\bLamb^i_T\circ \varepsilon_{(\bx_1,\dots,\bx_p)}^{+,p}]&=\E[\bLamb_T^i] + \sum_{j=1}^{p}\left(\bP^{i,k_j}(T-t_j,y_j)+ \int_{t_j}^{T} \bPs^{i,.}(T-w)\bP^{.,k_j}(w-t_j,y_j)dw\right).
	\end{align*}
\end{corollary}


\subsection{A general correlation formula}
As presented at the beginning of this section, our methodology allows one to compute more general functionals of the process $\boldsymbol{Z}$ which would not be possible by using existing methodologies (as the one presented in \cite{Bacryetal2013} with Volterra equations or in \cite{jaisson2014limit_appendix} with the moment measures). In particular one would like to compute the expectation of building blocks having the form of a product $\boldsymbol{Z}\boldsymbol{\Gamma}$ where $\boldsymbol{Z}$ is an $(\bZeta,\bM,\bP)-$MSPD and $\boldsymbol{\Gamma}:= \left(\boldsymbol{\Gamma}^\ell\right)_{1\leq \ell \leq d} \in L^2(\Omega)$ and whose expectation of the shifted version can be written as  
$$\E\left[ \boldsymbol{\Gamma}^\ell\circ\varepsilon_{(\bx_1,\dots,\bx_p)}^{+,p}\right] =\E\left[\boldsymbol{\Gamma}^\ell\right] +  \sum_{j=1}^p \boldsymbol{\rho}_\ell(k_j,t_j,y_j),$$
with  $\boldsymbol{\rho}:= \left(\boldsymbol{\rho}_\ell\right)_{1\leq \ell \leq d}\in \mathcal M_{d,1}^+$, $\boldsymbol{\rho}_\ell:\{1,\dots,d\}\times\R_+^2 \to \R_+$ and such that 
\begin{equation}
\label{eq:integrkernelstill}
\int_{\R_+}\bZeta^{i,k}(T-t,y)\boldsymbol{\rho}_\ell(k,t,y)\bnu^{k}(dy)<+\infty, \; \forall (i,k,t).
\end{equation}
We define  for $T >0$,  $(\bZeta \boldsymbol{\rho}_\ell)^T: \R_+ \times\R_+\to\mathcal M_{d,d}^+$ and $\overline{(\bZeta \boldsymbol{\rho}_\ell)^T}: \R^+ \to \mathcal M_{d,d}^+ $  as
$$
(\bZeta \boldsymbol{\rho}_\ell)^T(t,y):= \left(\bZeta^{i,k}(T-t,y)\boldsymbol{\rho}_\ell(k,t,y)\right)_{1\leq i,k \leq d},
$$
$$
\overline{(\bZeta \boldsymbol{\rho}_\ell)^T}(t) := \left(\int_{\R_+}\bZeta^{i,k}(T-t,y)\boldsymbol{\rho}_\ell(k,t,y)\bnu^{k}(dy)\right)_{1\leq i,k \leq d}.
$$
Moreover, since $(\bZeta \boldsymbol{\rho}_\ell)^T$ is a \textit{d-}kernel (due to the integrability Condition (\ref{eq:integrkernelstill})) and  $\boldsymbol{Z}^{(\boldsymbol{\bZeta\rho}_\ell)^T}$ is a $((\bZeta\boldsymbol{\rho}_\ell)^T,\bM,\bP)-$MSPD, its expectation is given by Proposition \ref{prop:expectation} as
\begin{align*}
	\E[\boldsymbol{Z}^{(\boldsymbol{\bZeta\rho}_\ell)^T}_T]= \int_{0}^{T}\overline{(\bZeta\boldsymbol{\rho}_\ell)^T}(u)\left(\int_{0}^{u}\bPs(u-v)\bM(v)dv + \bM(u) \right)du.
\end{align*}
\textbf{Convention:}
To lighten the notations, we write $\bZeta \boldsymbol{\rho}_\ell:=(\bZeta \boldsymbol{\rho}_\ell)^T$ when there is no ambiguity in the context.\\
In the result below, $T>0$ is a fixed horizon.

\begin{theorem}\label{prop:expGamma}\label{th:gamma}
	Let $\boldsymbol{Z}$ a $(\bZeta,\bM,\bP)-$MSPD and $(\bx_1,\dots,\bx_p)\in (\bX_T)^p$ satisfying the compatibility condition (\ref{eq:compatibility}). Let $\boldsymbol{\Gamma} \in \mathcal M_{d,1}^+$ such that for $1\leq \ell \leq d$, $\boldsymbol{\Gamma}^\ell\in L^2(\Omega)$ and
	\begin{equation}\label{relationgamma}
		\E\left[ \boldsymbol{\Gamma}^\ell\circ\varepsilon_{(\bx_1,\dots,\bx_p)}^{+,p}\right] =\E\left[\boldsymbol{\Gamma}^\ell\right] +  \sum_{j=1}^p \boldsymbol{\rho}_\ell(k_j,t_j,y_j),
	\end{equation}
	where $\boldsymbol{\rho}_\ell:\{1,\dots,d\}\times\R_+^2 \to \R_+$ is a deterministic function (specific to $\boldsymbol{\Gamma}^\ell$) satisfying (\ref{eq:integrkernelstill}).\\
	Then for $1\leq i,\ell \leq d$ we have
	$$\E\left[\myZ{T}\boldsymbol{\Gamma}^\ell\right] =\E\left[\myZ{T}\right] \E\left[\boldsymbol{\Gamma}^\ell\right] + \E\left[(\boldsymbol{Z}^{\boldsymbol{\bZeta\rho}_\ell}_T)^i\right] + \int_{0}^T\bZetaE^{i,.}(T-v)\left(\int_{0}^{v}\bPs(v-w)\E\left[\boldsymbol{Z}^{\boldsymbol{\bP\rho_\ell}\hspace{0.05cm}}_w\right]dw + \E\left[\boldsymbol{Z}^{\boldsymbol{\bP\rho_\ell}\hspace{0.05cm}}_v\right] \right)dv,$$
	where $\boldsymbol{Z^{\bZeta \rho_\ell}}$ is a $(\bZeta \boldsymbol{\rho}_\ell,\bM,\bP)-$MSPD and $\boldsymbol{Z^{\bP\rho_\ell}}$ is a $(\bP\boldsymbol{\rho}_\ell,\bM,\bP)-$MSPD.
\end{theorem}
\noindent The demonstration  of Theorem \ref{prop:expGamma} relies on the following lemma, whose proof is postponed in the appendix.
\begin{lemma}\label{lemma:magic}
	Let $\boldsymbol{Z}$ a $(\bZeta,\bM,\bP)-$MSPD, and $\boldsymbol{\rho}_\ell:\{1,\dots,d\}\times\R_+^2 \to \R_+$ such that $\int_{\R_+}\boldsymbol{\rho}_\ell(k,t,y)\nu^k(dy)<+\infty$, then  for $1\leq i \leq d$
	\begin{align*}
		&\sum_{n=1}^{+\infty} \frac{1}{n!} \int_{\bX_T^n}\left(\T^n_{(\bx_1,\dots,\bx_n)} \myZ{T}\right)\;\sum_{j=1}^n\boldsymbol{\rho}_\ell{(k_j,t_j,y_j)}\; \br^{\otimes n}(d\bx_1,\ldots,d\bx_n) \\
		&=\E\left[(\boldsymbol{Z}^{\boldsymbol{\bZeta\rho}_\ell}_T)^i\right] +\int_{0}^T\bZetaE^{i,.}(T-v)\int_{0}^{v}\left(\bPs(v-w)\E[\boldsymbol{Z}^{\boldsymbol{\bP\rho_\ell}}_{w}] + \E[\boldsymbol{Z}^{\boldsymbol{\bP\rho_\ell}}_{v}] \right)dwdv. 
	\end{align*}
	Moreover if $\sum \boldsymbol{\rho}_\ell \equiv 1$, then
	\begin{align*}
		&\sum_{n=1}^{+\infty} \frac{1}{n!} \int_{\bX_T^n} \T^n_{(\bx_1,\dots,\bx_n)} \myZ{T} \; \br^{\otimes n}(d\bx_1,\ldots,d\bx_n) = \E\left[\boldsymbol{Z}^i_T\right].
	\end{align*} 
\end{lemma}
\begin{proof}[Proof of Theorem \ref{prop:expGamma}]
	Using successively the pseudo-chaotic expansion for $\myZ{T}$, the Mecke formula and  Lemma \ref{lemma:magic} we have 
	\begin{align*}
		&\E[\myZ{T}\boldsymbol{\Gamma}^\ell]=  \lim_{M\to \infty}\sum_{n=1}^{+\infty}\frac{1}{n!}\E[\boldsymbol{\mathcal{I}}_n(\T^n \myZ{T}\ind{([0,T]\times[0,M])^n})\boldsymbol{\Gamma}^\ell]\\
		&= \lim_{M\to \infty}\sum_{n=1}^{+\infty}\frac{1}{n!} \int_{(\bX_T^M)^n} \T^n_{(\bx_1,\dots,\bx_n)} \myZ{T}\E\left[ \boldsymbol{\Gamma}^\ell\circ\varepsilon_{(\bx_1,\dots,\bx_n)}^{+,n}\right] \br^{\otimes n}(d\bx_1,\ldots,d\bx_n) \\
		&=\E[\myZ{T}]\E[\boldsymbol{\Gamma}^\ell] +\sum_{n=1}^{+\infty} \frac{1}{n!} \int_{\bX_T^n}\T^n_{(\bx_1,\dots,\bx_n)} \myZ{T}\sum_{j=1}^n\boldsymbol{\rho}_\ell{(t_j,y_j,k_j)}\br^{\otimes n}(d\bx_1,\ldots,d\bx_n)\\
		&= \E[\myZ{T}]\E[\boldsymbol{\Gamma}^\ell] + \E\left[(\boldsymbol{Z}^{\boldsymbol{\bZeta\rho}_\ell}_T)^i\right] +\int_{0}^T\bZetaE^{i,.}(T-v)\left(\int_{0}^{v}\bPs(v-w)\E\left[\boldsymbol{Z}^{\boldsymbol{\bP\rho_\ell}}_{w}\right]dw + \E\left[\boldsymbol{Z}^{\boldsymbol{\bP\rho_\ell}}_{v}\right] \right)dv
	\end{align*}
	where  $X^{M}_T:=\{1,\dots,d\}\times[0,T]\times[0,M]\times\R_+$.
\end{proof}
\noindent Theorem \ref{prop:expGamma} is now applied to compute the covariance of two MSPDs having different kernels.
\subsection{Correlations of MSPDs}
By combining Proposition \ref{prop:expectation} and  Theorem \ref{prop:expGamma}, we deduce the following proposition.
\begin{theorem}[Correlations of two MSPDs]\label{prop:correlationZ}
	Let $\boldsymbol{Z}$ a $(\bZeta,\bM,\bP)-$MSPD and $\boldsymbol{\tilde Z}$ a $(\tilde\bZeta,\tilde\bM,\tilde\bP)-$MSPD, then for $0 \leq T \leq S$, their covariance (for $1\leq i, \ell \leq d$) is 
	\begin{equation}\label{eq:covMSPD}
		Cov\left(\boldsymbol{Z}^i_T, \boldsymbol{\tilde Z}^\ell_S\right) = \E\left[(\boldsymbol{Z}_T^{\boldsymbol{\bZeta\tilde\rho_\ell}})^i\right] + \int_{0}^T\bZetaE^{i,.}(T-v)\left(\int_{0}^{v}\bPs(v-w)\E[\boldsymbol{Z}^{\boldsymbol{\bP\tilde\rho_\ell}}_{w}]dw + \E[\boldsymbol{Z}^{\boldsymbol{\bP\tilde\rho_\ell}}_{v}] \right)dv,
	\end{equation}
	where  $\bZeta \boldsymbol{\tilde\rho_\ell} := (\bZeta \boldsymbol{\tilde\rho_\ell})^T$ and $\boldsymbol{Z^{\bZeta\tilde\rho_\ell}}$ is a $(\bZeta \boldsymbol{\tilde\rho_\ell},\bM,\bP)-$MSPD $; \\{\bP\boldsymbol{\tilde\rho_\ell}}:={(\bP\boldsymbol{\tilde\rho_\ell})^T}$
	and $\boldsymbol{Z^{\bP\boldsymbol{\tilde\rho_\ell}}}$ a $(\bP \boldsymbol{\tilde\rho_\ell},\bM,\bP)-$MSPD with
	$$
	\boldsymbol{\tilde\rho}_\ell{(k,t,y)}:=\tilde\bZeta^{\ell,k}(S-t,y)
	+ \int_{t}^{S}\overline{\tilde\bZeta}^{\ell,.}(S-v) \left(\tilde\bP^{.,k}(v-t,y) + \int_{t}^{v}\TbPs(v-w) \tilde\bP^{.,k}(w-t,y)dw \right)dv.
	$$
\end{theorem}

\begin{proof}
	Let $p \in \N^*$, $(\bx_1 \dots, \bx_p)\in (\bX_T)^p$ satisfying the compatibility condition (\ref{eq:compatibility}), then from Proposition \ref{prop:expectation} we have that
	\begin{align*}
		\E[\boldsymbol{\tilde Z}_S^\ell\circ \varepsilon_{(\bx_1,\dots,\bx_p)}^{+,p}]&=\E[\boldsymbol{\tilde Z}_S^\ell] +  \sum_{j=1}^{p} \tilde\bZeta^{\ell,k_j}(S-t_j,y_j)\\
		&+\sum_{j=1}^{p} \int_{t_j}^{S}\overline{\tilde\bZeta}^{\ell,.}(S-v) \left(\tilde\bP^{.,k_j}(v-t_j,y_j) + \int_{t_j}^{v}\TbPs(v-w) \tilde\bP^{.,k_j}(w-t_j,y_j)dw \right)dv\\
		&=\E[\boldsymbol{\tilde Z}_S^\ell]  + \sum_{j=1}^{p}\boldsymbol{\tilde\rho}_\ell {(k_j,t_j,y_j)}.
	\end{align*}
	Thus applying Theorem \ref{prop:expGamma} with $\boldsymbol{\Gamma}^\ell:=\boldsymbol{\tilde Z}^\ell$ gives
	\begin{align*}
		Cov\left(\boldsymbol{Z}^i_T,\boldsymbol{\tilde Z}^\ell_S\right) = \E\left[(\boldsymbol{Z}_T^{\boldsymbol{\bZeta\tilde\rho_\ell}})^i\right] + \int_{0}^T\bZetaE^{i,.}(T-v)\left(\int_{0}^{v}\bPs(v-w)\E[\boldsymbol{Z}^{\boldsymbol{\bP\tilde\rho_\ell}}_{w}]dw + \E[\boldsymbol{Z}^{\boldsymbol{\bP\tilde\rho_\ell}}_{v}] \right)dv,
	\end{align*}
	where $\boldsymbol{Z^{\bZeta\tilde\rho_\ell}}$ is a $(\bZeta \boldsymbol{\tilde\rho_\ell},\bM,\bP)-$MSPD and $\boldsymbol{Z^{\bP\boldsymbol{\tilde\rho_\ell}}}$ a $(\bP \boldsymbol{\tilde\rho_\ell},\bM,\bP)-$MSPD.
\end{proof}
\noindent Remark that the  second term  in  the right hand side  of (\ref{eq:covMSPD}) can also be written  as the expectation of a $(\bZeta,\E[\boldsymbol{Z}^{\boldsymbol{\bP\tilde\rho_\ell}}],\bP)-$MSPD. More generally, 
this procedure, used here to compute the covariance, can be easily iterated to compute moments of further orders, to the price of cumbersome expressions.

\subsection{Case of a counting process with separable kernel}
In the case of counting process with separable \textit{d-}kernel $\bP(t,y) = \boldsymbol{\Phi}(t) \star\boldsymbol{B}(y)$, the expression of the covariance of the process at two different dates can be simplified, thanks to an extra convolution in the final expression.  For $1\leq i, \ell \leq d$, we introduce 
the transposed  \textit{d-}dimensional vector $\boldsymbol{\mathfrak{C}}$
\begin{equation}\label{eq:correlseparable}
(\boldsymbol{\mathfrak{C}}^k)_{1\leq k \leq d}:=\left(\frac{\int_{\R_+}\boldsymbol{B}^{i,k}(y) \boldsymbol{B}^{\ell,k}(y)\bnu^{k}(dy)}{\int_{\R_+}\boldsymbol{B}^{i,k}(y)\bnu^{k}(dy)\int_{\R_+}\boldsymbol{B}^{\ell,k}(y)\bnu^{k}(dy)}\right)_{1\leq k \leq d}
\hspace*{-5mm}=\left(\frac{\mathbb E( \boldsymbol{B}^{i,k}(Y^k) \boldsymbol{B}^{\ell,k}(Y^k) ) }{\mathbb E(\boldsymbol{B}^{i,k}(Y^k)) \mathbb E (\boldsymbol{B}^{\ell,k}(Y^k) )}\right)_{1\leq k \leq d}.
\end{equation}
Remark  that $\boldsymbol{\mathfrak{C}}^{k} = 1$
	means that $Cov(\boldsymbol{B}^{i,k}(Y^k),\boldsymbol{B}^{\ell,k}(Y^k))=0$.

\begin{prop}\label{prop:cov_sep}
	Let $\boldsymbol{H}$ a $(\boldsymbol{\bI_d},\bM,\bP)-$MSPD counting process and $\bP$ a separable \textit{d-}kernel such that $\bP(t,y) = \boldsymbol{\Phi}(t) \star\boldsymbol{B}(y)$, then for $0 \leq T \leq S$  and  $1\leq i, \ell \leq d$ we have 
	\begin{align*}
		&Cov\left(\boldsymbol{H}^i_T,\boldsymbol{H}^\ell_S\right)\\
		&=\int_{0}^T\int_{u}^{T}\bPs^{i,.}(t-u)dt\star\left( \boldsymbol{\bI_d}^{\ell,.}+ \boldsymbol{\mathfrak{C}}\star\int_{u}^{S}\bPs^{\ell,.}(v-u)dv\right)\left(\int_{0}^{u}\bPs(u-v)\bM(v)dv + \bM(u) \right)du\\
		&+ \int_{0}^{T}\boldsymbol{\bI_d}^{i,.}\star\left(\boldsymbol{\bI_d}^{\ell,.}+ \int_{u}^{S}\bPs^{\ell,.}(v-u)dv\right)\left(\int_{0}^{u}\bPs(u-v)\bM(v)dv + \bM(u) \right)du,
	\end{align*}
	where the transposed  \textit{d-}dimensional vector $\boldsymbol{\mathfrak{C}}$ is given in \eqref{eq:correlseparable}. \\
	 Moreover, if $\boldsymbol{\mathfrak{C}}^{k} = 1$ for all $1\leq k \leq d$
	 the expression simplifies as 
		\begin{align*}
			&Cov\left(\boldsymbol{H}^i_T,\boldsymbol{H}^\ell_S\right)=\\
			&\hspace{-1.5cm}\int_{0}^T\left(\boldsymbol{\bI_d}^{i,.} + \int_{u}^{T}\bPs^{i,.}(y-u)dy\right)\star\left( \boldsymbol{\bI_d}^{\ell,.}+ \int_{u}^{S}\bPs^{\ell,.}(v-u)dv\right)\left(\int_{0}^{u}\bPs(u-v)\bM(v)dv + \bM(u) \right)du.
		\end{align*}
		
\end{prop}
\begin{proof}
Applying 	Theorem  \ref{prop:correlationZ}  for  $\bZeta = \boldsymbol{\bI_d}$ ($\boldsymbol{H}$ is a counting process) 
	\begin{equation}\label{eq:covcounting}
		Cov\left(\boldsymbol{H}^i_T,\boldsymbol{H}^\ell_S\right) = \E[(\boldsymbol{Z}^{\boldsymbol{\bZeta\rho_\ell}}_T)^i] + \int_{0}^T\boldsymbol{\bI_d}^{i,.}\left(\int_{0}^{v}\bPs(v-w)\E[\boldsymbol{Z}^{\boldsymbol{\bP\rho_\ell}}_{w}]dw + \E[\boldsymbol{Z}^{\boldsymbol{\bP\rho_\ell}}_{v}] \right)dv,
	\end{equation}
where $\boldsymbol{\bZeta\rho_\ell}:=(\boldsymbol{\bZeta\rho_\ell})^T $ is the diagonal matrix whose  diagonal elements are the components of the vector  $\boldsymbol{\rho_\ell}$  given below,
	 $\boldsymbol{Z^{\bZeta\rho_\ell}}$ is a $(\boldsymbol{\bZeta\rho_\ell},\bM,\bP)-$MSPD ,  $\boldsymbol{Z}^{\boldsymbol{\bP\rho_\ell}}$ is a $(\boldsymbol{\bP\rho_\ell},\bM,\bP)-$MSPD and 
	\begin{align*}
		\boldsymbol{\rho_\ell}(k,t,y)&=\ind{\ell=k}
		+ \int_{t}^{S}\boldsymbol{\bI_d}^{\ell,.} \left(\bP^{.,k}(v-t,y) + \int_{t}^{v}\bPs(v-w) \bP^{.,k}(w-t,y)dw \right)dv\\
		&= \ind{\ell=k}
		+ \int_{t}^{S}\left(\bP^{\ell,k}(v-t,y) + \int_{t}^{v}\bPs^{\ell,.}(v-w) \bP^{.,k}(w-t,y)dw \right)dv.
	\end{align*}
	The proof is  divided in four  steps,  first calculating in steps 1 and 2 the expectation of the corresponding MSPDs  $\E[\boldsymbol{Z}^{\boldsymbol{\bZeta\rho_\ell}}_T]$ and $\E[\boldsymbol{Z}^{\boldsymbol{\bP\rho_\ell}}_{w}]$,  then step 3 computes $\int_{0}^{v}\bPs(v-w)\E[\boldsymbol{Z}^{\boldsymbol{\bP\rho_\ell}}_{w}]dw$ and finally step 4  gathers the previous expressions  to give the final result.\\
Steps 1 and 2 rely on the assumption of  a  separable excitation kernel $\bP(t,y)= \boldsymbol{\Phi}(t)\star\boldsymbol{B}(y)$. We recall the notation $\overline{\boldsymbol{B}}^{\ell,k}:= \int_{\R_+}\boldsymbol{B}^{\ell,k}(y)\bnu^{k}(dy)$.
	In what follows, the ratio $\frac{\boldsymbol{B}^{\ell,k}(y)}{\boldsymbol{\overline{B}}^{\ell,k}}$ represents   the relative value of the outcome   $\boldsymbol{B}^{\ell,k}(y)$ with respect  to its mean value $\boldsymbol{\overline{B}}^{\ell,k}$.\\

\noindent \underline{Step 1:} Computing  $\E[(\boldsymbol{Z}^{\boldsymbol{\bZeta\rho_\ell}}_T)^i]$.\\
	By  putting the relative quantity  $\frac{\boldsymbol{B}^{\ell,k}(y)}{\boldsymbol{\overline{B}}^{\ell,k}}$ in factor (thanks to the separability of the kernel) and then by using the convolution relation \eqref{eq:magic_convol}
	\begin{align*}
	&	\boldsymbol{\bZeta\rho_\ell}^{i,k}(u,y)=\ind{i=k} \left(\ind{\ell=k} +\int_{u}^{S} \left(\bP^{\ell,k}(v-u,y) + \int_{u}^{v}\bPs^{\ell,.}(v-w) \bP^{.,k}(w-u,y)dw \right)dv\right)\\
		&=\ind{i=k} \left(\ind{\ell=k} +\int_{u}^{S} \left(\bPE^{\ell,k}(v-u)\frac{\boldsymbol{B}^{\ell,k}(y)}{\boldsymbol{\overline{B}}^{\ell,k}} + \left[\int_{u}^{v}\bPs(v-w) \bPE(w-u)dw\right]^{\ell,k}\frac{\boldsymbol{B}^{\ell,k}(y)}{\boldsymbol{\overline{B}}^{\ell,k}}  \right)dv\right)\\
		&= \ind{i=k} \left(\ind{\ell=k} +\frac{\boldsymbol{B}^{\ell,k}(y)}{\boldsymbol{\overline{B}}^{\ell,k}}\int_{u}^{S}\bPs^{\ell,k}(v-u)dv\right); \quad u \leq T.
	\end{align*}
Integrating with respect to  $\bnu^{k}(dy)$  then yields
	$$
	\overline{\boldsymbol{\bZeta\rho_\ell}}^{i,k}(u) = \ind{i=k} \left(\ind{\ell=k} +\int_{u}^{S}\bPs^{\ell,k}(v-u)dv\right),  \quad u \leq T,
	$$
and using Proposition \ref{prop:expectation} for a $(\boldsymbol{\bZeta\rho_\ell},\bM,\bP)-$MSPD 
	\begin{align*}
		\E[(\boldsymbol{Z}^{\boldsymbol{\bZeta\rho_\ell}}_T)^i]= \int_{0}^{T}\boldsymbol{\bI_d}^{i,.}\star\left(\boldsymbol{\bI_d}^{\ell,.}+ \int_{u}^{S}\bPs^{\ell,.}(v-u)dv\right)\left(\int_{0}^{u}\bPs(u-v)\bM(v)dv + \bM(u) \right)du.
	\end{align*}
\underline{Step 2:} Computing $\E[(\boldsymbol{Z}^{\boldsymbol{\bP\rho_\ell}}_T)^i]$.\\
	By repeating the same methodology as in Step 1, we have 
	\begin{align*}
		\boldsymbol{\bP\rho_\ell}^{i,k}(u,y)
		&= \bP^{i,k}(T-u,y) \left(\ind{\ell=k} + \frac{\boldsymbol{B}^{\ell,k}(y)}{\boldsymbol{\overline{B}}^{\ell,k}}\int_{u}^{S}\bPs^{\ell,k}(v-u)dv\right).
	\end{align*}
	Integrating with respect to  $\bnu^{k}(dy)$ yields 
	\begin{align*}
		\overline{\boldsymbol{\bP \rho_\ell}}^{i,k}(u) &= \bPE^{i,k}(T-u)\ind{\ell=k} + \int_{\R_+}\boldsymbol{\Phi}^{i,k}(T-u)\boldsymbol{B}^{i,k}(y)\frac{\boldsymbol{B}^{\ell,k}(y)}{\boldsymbol{\overline{B}}^{\ell,k}}\int_{u}^{S}\bPs^{\ell,k}(v-u)dv \nu^{k}(dy)\\
		&= \bPE^{i,k}(T-u)\ind{\ell=k} + \bPE^{i,k}(T-u)\int_{u}^{S}\bPs^{\ell,k}(v-u)dv \int_{\R_+}\frac{\boldsymbol{B}^{i,k}(y)}{\overline{\boldsymbol{B}}^{i,k}}\frac{\boldsymbol{B}^{\ell,k}(y)}{\boldsymbol{\overline{B}}^{\ell,k}}\nu^{k}(dy)\\
		&= \bPE^{i,k}(T-u)\left(\ind{\ell=k} +     \boldsymbol{\mathfrak{C}}^k  \int_{u}^{S}\bPs^{\ell,k}(v-u)dv\right)
	\end{align*}
that is $$
	\overline{\bP\rho}^{i,.}(u)=\bPE^{i,.}(T-u)\star\left( \boldsymbol{\bI_d}^{\ell,.}+ \boldsymbol{\mathfrak{C}}\star\int_{u}^{S}\bPs^{\ell,.}(v-u)dv\right),
	$$
	and  using Proposition \ref{prop:expectation} for a $(\boldsymbol{\bP\rho_\ell},\bM,\bP)-$MSPD
	\begin{align*}
		\E[(\boldsymbol{Z}^{\boldsymbol{\bP\rho_\ell}}_T)^i]= \int_{0}^{T}\bPE^{i,.}(T-u)\star\left( \boldsymbol{\bI_d}^{\ell,.}+ \boldsymbol{\mathfrak{C}}\star\int_{u}^{S}\bPs^{\ell,.}(v-u)dv\right)\left(\int_{0}^{u}\bPs(u-v)\bM(v)dv + \bM(u) \right)du.
	\end{align*}
	\underline{Step 3:} Computing $\int_{0}^{t}\bPs(t-w)\E[\boldsymbol{Z}^{\boldsymbol{\bP\rho_\ell}}_{w}]dw $.\\
	Using Step 2 and again the  convolution relation \eqref{eq:magic_convol}, we have
	\begin{align*}
		&{\sum_{k=1}^{d}\int_{0}^{t}\bPs^{j,k}(t-w)\E[(\boldsymbol{Z}^{\boldsymbol{\bP\rho_\ell}}_w)^k]dw}\\
		&\hspace{-1.2cm}=\sum_{k=1}^{d}\int_{0}^{t}\bPs^{j,k}(t-w)\int_{0}^{w}\bPE^{k,.}(w-u)\star\left( \boldsymbol{\bI_d}^{\ell,.}+ \boldsymbol{\mathfrak{C}}\star\int_{u}^{S}\bPs^{\ell,.}(v-u)dv\right)\left(\int_{0}^{u}\bPs(u-v)\bM(v)dv + \bM(u) \right)dudw\\
		&\hspace{-1.2cm}=\int_{0}^{t}\left[\int_{u}^{t}\bPs(t-w)\bPE(w-u)dw\right]^{j,.}\star\left( \boldsymbol{\bI_d}^{\ell,.}+ \boldsymbol{\mathfrak{C}}\star\int_{u}^{S}\bPs^{\ell,.}(v-u)dv\right)\left(\int_{0}^{u}\bPs(u-v)\bM(v)dv + \bM(u) \right)du\\
		&\hspace{-1.2cm}=\int_{0}^{t}\left(\bPs^{j,.}(t-u)-\bPE^{j,.}(t-u)\right)\star\left( \boldsymbol{\bI_d}^{\ell,.}+ \boldsymbol{\mathfrak{C}}\star\int_{u}^{S}\bPs^{\ell,.}(v-u)dv\right)\left(\int_{0}^{u}\bPs(u-v)\bM(v)dv + \bM(u) \right)du\\
		&\hspace{-1.2cm}=\int_{0}^{t}\bPs^{j,.}(t-u)\star\left( \boldsymbol{\bI_d}^{\ell,.}+ \boldsymbol{\mathfrak{C}}\star\int_{u}^{S}\bPs^{\ell,.}(v-u)dv\right)\left(\int_{0}^{u}\bPs(u-v)\bM(v)dv + \bM(u) \right)du - \E[(\boldsymbol{Z^{\bP\rho_\ell}_t})^{j}].
	\end{align*}
\underline{Step 4:} Final result. 	\\
Coming back to \eqref{eq:covcounting}, and integrating the  results of the previous steps,
	\begin{align*}
		&Cov\left(\boldsymbol{H}^i_T,\boldsymbol{H}^\ell_S\right)- \E[(\boldsymbol{Z}^{\boldsymbol{\bZeta\rho_\ell}}_T)^i]= \int_{0}^T\boldsymbol{\bI_d}^{i,.}  \left(\int_{0}^{t}\bPs(t-w)  E[\boldsymbol{Z}^{\boldsymbol{\bP\rho_\ell}}_{w}]dw + \E[\boldsymbol{Z}^{\boldsymbol{\bP\rho_\ell}}_{t}]\right)dt\\
		&\hspace{-1.5cm}=\sum_{j=1}^{d}\int_{0}^T\boldsymbol{\bI_d}^{i,j}  \int_{0}^{t}\bPs^{j,.}(t-u)\star\left( \boldsymbol{\bI_d}^{\ell,.}+ \boldsymbol{\mathfrak{C}}\star\int_{u}^{S}\bPs^{\ell,.}(v-u)dv\right)\left(\int_{0}^{u}\bPs(u-v)\bM(v)dv + \bM(u) \right)dudt\\
		&\hspace{-1.5cm}=\int_{0}^T\int_{0}^{t}\sum_{j=1}^{d}\ind{i=j}\bPs^{j,.}(t-u)\star\left( \boldsymbol{\bI_d}^{\ell,.}+ \boldsymbol{\mathfrak{C}}\star\int_{u}^{S}\bPs^{\ell,.}(v-u)dv\right)\left(\int_{0}^{u}\bPs(u-v)\bM(v)dv + \bM(u) \right)dudt\\
		&\hspace{-1.5cm}=\int_{0}^T\int_{u}^{T}\bPs^{i,.}(t-u)dt\star\left( \boldsymbol{\bI_d}^{\ell,.}+ \boldsymbol{\mathfrak{C}}\star\int_{u}^{S}\bPs^{\ell,.}(v-u)dv\right)\left(\int_{0}^{u}\bPs(u-v)\bM(v)dv + \bM(u) \right)du.
	\end{align*}       
By replacing the expression of $\E[(\boldsymbol{Z}^{\boldsymbol{\bZeta\rho_\ell}}_T)^i]$ computed in Step 1, we get 
	\begin{align*}
		&Cov\left(\boldsymbol{H}^i_T,\boldsymbol{H}^\ell_S\right)\\
		&=\int_{0}^T\int_{u}^{T}\bPs^{i,.}(t-u)dt\star\left( \boldsymbol{\bI_d}^{\ell,.}+ \boldsymbol{\mathfrak{C}}\star\int_{u}^{S}\bPs^{\ell,.}(v-u)dv\right)\left(\int_{0}^{u}\bPs(u-v)\bM(v)dv + \bM(u) \right)du\\
		&+ \int_{0}^{T}\boldsymbol{\bI_d}^{i,.}\star\left(\boldsymbol{\bI_d}^{\ell,.}+ \int_{u}^{S}\bPs^{\ell,.}(v-u)dv\right)\left(\int_{0}^{u}\bPs(u-v)\bM(v)dv + \bM(u) \right)du.
	\end{align*}
\end{proof}

\section{Conclusion}
In this paper,  we presented a general method for calculating  different quantities related to multidimensional self-exciting processes with dependencies (MSPD). This class of processes encompasses several quantities useful  for risk assessment, such as counting, loss,  intensity processes and their shifted versions, in a framework of cross-dependencies and impact of the severity component on the frequency component.  The methodology relies on the Poisson imbedding, the pseudo-chaotic expansion and Malliavin calculus, and is illustrated here to compute explicit formula for the correlations. A forthcoming companion paper will be dedicated to further developments such as the computation of moments of  higher order, as well as the pricing of insurance contracts with underlying asset a MSPD (e.g. stop-loss contracts).

\section*{Appendix}
\begin{proof}[Proof of Lemma \ref{lemma:magic}]
The aim is to compute the following quantity
\begin{align*}
&A:=\sum_{n=1}^{+\infty} \frac{1}{n!} \int_{\bX_T^n} \T^n_{(\bx_1,\dots,\bx_n)} \myZ{T}\sum_{j=1}^n\boldsymbol{\rho}_\ell{(t_j,y_j,k_j)}\br^{\otimes n}(d\bx_1,\ldots,d\bx_n) \\
&=\sum_{n=1}^{+\infty}\int_{\Delta_T^n}\sum_{k_1=1}^d \dots \sum_{k_n=1}^d\int_{\R_+^2}\T^n_{(\bx_1,\dots,\bx_n)} \myZ{T}\sum_{j=1}^n\boldsymbol{\rho}_\ell{(t_j,y_j,k_j)}dt_1 d\theta_1 \bnu^{k_1}(dy_1)\cdots  dt_n d\theta_n \bnu^{k_n}(dy_n).
\end{align*}
where $\bx_i :=(k_i,t_i,\theta_i,y_i)$, $\Delta_T^n$  is  the simplex 
$$\Delta_T^n:=\{(\bx_1,\cdots,\bx_n)\in\bX^n,t_1 < \cdots< t_i<t_{i+1}<\cdots<t_n<T\}$$
so that for any symmetric map $f$, $\frac{1}{n!} \int_{[0,T]^n} f(t_1,\ldots,t_n) dt_1 \cdots dt_n = \int_{\Delta_T^n} f(t_1,\ldots,t_n) dt_1 \cdots dt_n$. 
The  proof is divided in 5 steps, by successively integrating  with respect to the $\theta_i$  (Step 1), then  the $y_i$  (Step 2), followed by the  $k_i$ (Step 3), and finally with respect to the $t_i$  (Step 4). The last step consists in summing all the previous  terms depending on $n$ (Step 5). \\
\underline{Step 1 :} Integration with respect to $\theta$.\\
	Since $\boldsymbol{\rho_\ell}$ does not depend on $\theta$, we integrate first $\T^n_{(\bx_1,\dots,\bx_n)} \myZ{T}$ with respect to the $\theta_i$ variables by making use of (\ref{eq:ckcounting}) and Remark \ref{remark:Tnlambda} we have
	\begin{align*}
		&\int_{\R^n_+} \T^n_{(\bx_{1}, \dots, \bx_{n})} \myZ{T} d\theta_1 \dots d\theta_n  \\
		&=\int_{\R^n_+} \bZeta^{i,k_n}(T-t_n,y_n)\T^{n-1}_{(\bx_1,\dots,\bx_{n-1})}\ind{\theta_{n}\leq \bLamb^{k_n}_{t_n}} d\theta_1 \dots d\theta_n  \\
		&= \bZeta^{i,k_n}(T-t_n,y_n) \int_{\R_+^{n-1}} \T^{n-1}_{(\bx_1,\dots,\bx_{n-1})}\bLamb^{k_n}_{t_n}d\theta_{1}\dots d\theta_{n-1}  \\
		&= \bZeta^{i,k_n}(T-t_n,y_n)\int_{\R_+^{n-1}}\bP^{k_n,k_{n-1}}(t_n-t_{n-1},y_{n-1})\T^{n-2}_{(\bx_1,\dots,\bx_{n-2})}\ind{\theta_{n-1} \leq \lambda_{t_{n-1}}^{k_{n-1}}}d\theta_1 \dots d\theta_{n-1}  \\
		&= \bZeta^{i,k_n}(T-t_n,y_n)\bP^{k_n,k_{n-1}}(t_n-t_{n-1},y_{n-1})\int_{\R_+^{n-2}}\T^{n-2}_{(\bx_1,\dots,\bx_{n-2})}\lambda_{t_{n-1}}^{k_{n-1}}d\theta_1 \dots d\theta_{n-2}.
	\end{align*}
	Hence by induction 
	\begin{align*}
		\int_{\R_+^n} \T^n_{(\bx_{1}, \dots, \bx_{n})} \myZ{T} d\theta_1 \dots d\theta_n =\bZeta^{i,k_n}(T-t_n,y_n) \prod_{m=2}^{n} \bP^{k_m,k_{m-1}}(t_m-t_{m-1},y_{m-1})\bM^{k_1}(t_1) =:F(\bold{k},\bold{t},\bold{y})
	\end{align*}
	with $\bold{k}:=k_1,\cdots,k_n$, $\bold{t}:=t_1,\cdots,t_n$, $\bold{y}:=y_1,\cdots,y_n$, and
\begin{align*}
&A= \sum_{n=1}^{+\infty}\int_{\Delta_T^n}\sum_{k_1=1}^d \dots \sum_{k_n=1}^d\int_{\R_+} F(\bold{k},\bold{t},\bold{y}) \sum_{j=1}^n\boldsymbol{\rho}_\ell{(t_j,y_j,k_j)}dt_1 \bnu^{k_1}(dy_1)\cdots  dt_n \bnu^{k_n}(dy_n) \\
\end{align*}
	\underline{Step 2:} Integration with respect to $y$. \\
	We  separately treat the cases where $\sum \boldsymbol{\rho}_\ell \equiv 1 $ and $\sum \boldsymbol{\rho}_\ell \not\equiv 1$, and we denote the corresponding values   respectively  $A^{\not\boldsymbol{\rho}}$ and $A^{\boldsymbol{\rho}}$.\\
	$\bullet$ If $\sum \boldsymbol{\rho}_\ell \equiv 1$,
	\begin{align*}
		&\int_{\R_+^n} \left[\bZeta^{i,k_n}(T-t_n,y_n)\prod_{m=2}^{n} \bP^{k_m,k_{m-1}}(t_m-t_{m-1},y_{m-1})\bM^{k_1}(t_1)\right] \bnu^{k_1}(dy_1)\dots\bnu^{k_n}(dy_n)\\
		&=\bZetaE^{i,k_n}(T-t_n)\prod_{m=2}^{n} \bPE^{k_m,k_{m-1}}(t_m-t_{m-1})\bM^{k_1}(t_1).
	\end{align*}
	So in that case 
	$$ A^{\not\boldsymbol{\rho}} = \sum_{n=1}^{+\infty}\int_{\Delta_T^n}\sum_{k_1=1}^d \dots \sum_{k_n=1}^d \bZetaE^{i,k_n}(T-t_n)\prod_{m=2}^{n} \bPE^{k_m,k_{m-1}}(t_m-t_{m-1})\bM^{k_1}(t_1) dt_1\cdots, dt_n.$$
	$\bullet$ If $\sum \boldsymbol{\rho}_\ell \not\equiv 1,$
	\begin{align*}
		&\hspace{-1.5cm}\int_{\R_+^n} \left[\bZeta^{i,k_n}(T-t_n,y_n)\prod_{m=2}^{n} \bP^{k_m,k_{m-1}}(t_m-t_{m-1},y_{m-1})\bM^{k_1}(t_1)\right]\sum_{j=1}^n\boldsymbol{\rho_\ell}{(t_j,y_j,k_j)} \bnu^{k_1}(dy_1) \dots\bnu^{k_n}(dy_n)\\
		&=\int_{\R_+^n} \left[\bZeta^{i,k_n}(T-t_n,y_n)\boldsymbol{\rho_\ell}{(t_n,y_n,k_n)}\prod_{m=2}^{n} \bP^{k_m,k_{m-1}}(t_m-t_{m-1},y_{m-1})\bM^{k_1}(t_1)\right]\bnu^{k_1}(dy_1)\dots\bnu^{k_n}(dy_n)\\
		&+\int_{\R_+^n} \left[\bZeta^{i,k_n}(T-t_n,y_n)\prod_{m=2}^{n} \bP^{k_m,k_{m-1}}(t_m-t_{m-1},y_{m-1})\bM^{k_1}(t_1)\right]\sum_{j=1}^{n-1}\boldsymbol{\rho_\ell}{(t_j,y_j,k_j)} \bnu^{k_1}(dy_1)\dots\bnu^{k_n}(dy_n)\\
		&= \overline{\bZeta\boldsymbol{\rho}_\ell}^{i,k_n}(t_n)\prod_{m=2}^{n} \bPE^{k_m,k_{m-1}}(t_m-t_{m-1})\bM^{k_1}(t_1)\\
		&+\sum_{j=1}^{n-1}\int_{\R_+^n}\bZeta^{i,k_n}(T-t_n,y_n)\bP^{k_{j+1},k_{j}}(t_{j+1}-t_{j},y_{j})\boldsymbol{\rho_\ell}{(t_j,y_j,k_j)}\\
		&\left( \prod_{m=2,m\neq j+1}^{n} \bP^{k_m,k_{m-1}}(t_m-t_{m-1},y_{m-1})\right)\bM^{k_1}(t_1)\bnu^{k_1}(dy_1) \dots\bnu^{k_n}(dy_n)\\
		&=  \overline{\bZeta\boldsymbol{\rho}_\ell}^{i,k_n}(t_n)\prod_{m=2}^{n} \bPE^{k_m,k_{m-1}}(t_m-t_{m-1})\bM^{k_1}(t_1)\\
		&+\sum_{j=1}^{n-1}\bZetaE^{i,k_n}(T-t_n)\overline{\bP\boldsymbol{\rho}_\ell}^{k_{j+1},k_{j}}(t_{j+1}-t_{j})\left( \prod_{m=2,m\neq j+1}^{n} \bPE^{k_m,k_{m-1}}(t_m-t_{m-1})\right)\bM^{k_1}(t_1)\\
		&=:F^{\boldsymbol{\rho}}(\bold{k},\bold{t})
	\end{align*}
	where by abuse of notation 
	$$ \overline{\bP\boldsymbol{\rho}_\ell}^{k_{j+1},k_{j}}(t_{j+1}-t_{j}):=   \int \bP^{k_{j+1},k_{j}}(t_{j+1}-t_{j}, y_j) \boldsymbol{\rho}_\ell (k_j, t_j, y_j  ) \bnu^{k_j}(dy_j).$$
So in that case 
\begin{align*}
	A^{\boldsymbol{\rho}} &= \sum_{n=1}^{+\infty}\int_{\Delta_T^n}\sum_{k_1=1}^d \dots \sum_{k_n=1}^d F^{\boldsymbol{\rho}}(\bold{k},\bold{t}) dt_1 \cdots dt_n.
\end{align*}
	\underline{Step 3:} Integration with respect to $k$.\\
	This step lies essentially in the identification of matrix products.\\
	$\bullet$  If $\sum \boldsymbol{\rho}_\ell \equiv 1$,
	\begin{align*}
		&\sum_{k_1=1}^d\dots\sum_{k_n=1}^d\bZetaE^{i,k_n}(T-t_n)\prod_{m=2}^{n} \bPE^{k_m,k_{m-1}}(t_m-t_{m-1})\bM^{k_1}(t_1)\\
		&=\sum_{k_n=1}^d\bZetaE^{i,k_n}(T-t_n)\left(\prod_{m=2}^{n}\bPE(t_m - t_{m -1}) \boldsymbol{\bM}(t_1)\; \right)^{k_n,.}\\
		&= \bZetaE^{i,.}(T-t_n)\prod_{m=2}^{n}\bPE(t_m - t_{m -1})  \boldsymbol{\bM}(t_1).
	\end{align*}
	Hence
$$ A^{\not\boldsymbol{\rho}} = \sum_{n=1}^{+\infty}\int_{\Delta_T^n}\bZetaE^{i,.}(T-t_n)\prod_{m=2}^{n}\bPE(t_m - t_{m -1})  \boldsymbol{\bM}(t_1) dt_1 \cdots  dt_n.$$
	$\bullet$  If $\sum \boldsymbol{\rho}_\ell \not\equiv 1$,
	\begin{align*}
		&\sum_{k_1=1}^d\dots\sum_{k_n=1}^d\overline{\bZeta\boldsymbol{\rho}_\ell}^{i,k_n}(t_n)\prod_{m=2}^{n} \bPE^{k_m,k_{m-1}}(t_m-t_{m-1})\bM^{k_1}(t_1)\\
		&+\sum_{k_1=1}^d\dots\sum_{k_n=1}^d\sum_{j=1}^{n-1}\bZetaE^{i,k_n}(T-t_n)
		\overline{\bP\boldsymbol{\rho}_\ell}^{k_{j+1},k_{j}}(t_{j+1}-t_{j}) 
		\left( \prod_{m=2,m\neq j+1}^{n} \bPE^{k_m,k_{m-1}}(t_m-t_{m-1})\right)\bM^{k_1}(t_1)\\
		&=\overline{\bZeta\boldsymbol{\rho}_\ell}^{i,.}(t_n)\prod_{m=2}^{n} \bPE(t_m-t_{m-1})\bM(t_1)\\
		&+\sum_{j=1}^{n-1}\bZetaE^{i,.}(T-t_n)\prod_{m = j+2}^n \bPE(t_m-t_{m-1})\overline{\bP\boldsymbol{\rho}_\ell}(t_{j+1}-t_{j}) \prod_{m = 2}^j\bPE(t_m-t_{m-1})\bM(t_1)\\
		&=:G^{\boldsymbol{\rho}}(\bold{t}).
	\end{align*}
	Hence, 
	$$ A^{\boldsymbol{\rho}} = \sum_{n=1}^{+\infty}\int_{\Delta_T^n}G^{\boldsymbol{\rho}}(\bold{t}) dt_1\dots dt_n.$$
	\underline{Step 4:} Integration with respect to $t$.\\
	The step	consists in identifying the convolution products that appear in the expressions.\\
	$\bullet$ If $\sum \boldsymbol{\rho}_\ell \equiv 1$,
$$\int_{\Delta_T^n}\bZetaE^{i,.}(T-t_n)\prod_{m=2}^{n}\bPE(t_m - t_{m -1})  \boldsymbol{\bM}(t_1) dt_1 \dots dt_n=\int_0^T \int_{t_1}^T\bZetaE^{i,.}(T-t_n)\bPE_{n-1}(t_n - t_{1})\boldsymbol{\bM}(t_1)dt_n dt_1.$$
	Hence,
	$$ A^{\not\boldsymbol{\rho}} = \sum_{n=1}^{+\infty}\int_0^T \int_{w}^T\bZetaE^{i,.}(T-v)\bPE_{n-1}(v - w)\boldsymbol{\bM}(w)dv dw.$$
	$\bullet$ If $\sum \boldsymbol{\rho}_\ell \not\equiv 1$,\\
	the first term in $G^{\boldsymbol{\rho}}(\bold{t})$ integrates as above:
	$$\int_{\Delta_T^n}\overline{\bZeta\boldsymbol{\rho}_\ell}^{i,.}(t_n)\prod_{m=2}^{n} \bPE(t_m-t_{m-1})\bM(t_1)dt_1 \dots dt_n
	=\int_0^T \int_{w}^T \overline{\bZeta\boldsymbol{\rho}_\ell}^{i,.}(v) \bPE_{n-1}(v-w)\bM(t_1) dv dw$$
	where we recall that
		$$ \overline{\bZeta\boldsymbol{\rho}_\ell}^{k_{j+1},k_{j}}(v):=   \int \bZeta^{k_{j+1},k_{j}}(T-v, y_j) \boldsymbol{\rho}_\ell (k_j, v, y_j  ) \bnu^{k_j}(dy_j).$$
	For the second term
	\begin{align*}
		&\int_{\Delta_T^n}\sum_{j=1}^{n-1}\bZetaE^{i,.}(T-t_n)\prod_{m = j+2}^n \bPE(t_m-t_{m-1})\overline{\bP\boldsymbol{\rho}_\ell}(t_{j+1}-t_{j}) \prod_{m = 2}^j\bPE(t_m-t_{m-1})\bM(t_1)dt_1 \dots dt_n\\
		&=\sum_{j=1}^{n-1}\int_{0}^{T}\bZetaE^{i,.}(T-t_n)\int_{0}^{t_n}\cdots\int_{0}^{t_{j+2}}\prod_{m = j+2}^n \bPE(t_m-t_{m-1})\int_{0}^{t_{j+1}}\overline{\bP\boldsymbol{\rho}_\ell}(t_{j+1}-t_{j})\int_{0}^{t_{j}}\dots\\
		&\hspace*{1cm}\int_{0}^{t_{2}} \prod_{m = 2}^j\bPE(t_m-t_{m-1})\bM(t_1)dt_1\cdots dt_n\\
		&=\sum_{j=1}^{n-1}\int_{0}^{T}\int_{t_j}^T\bZetaE^{i,.}(T-t_n)\int_{t_j}^{t_{n}}\bPE_{n-j-1}(t_n-t_{j+1})\overline{\bP\boldsymbol{\rho}_\ell}(t_{j+1}-t_{j}) \int_{0}^{t_{j}}\bPE_{j-1}(t_j-t_{1})\bM(t_1)dt_{1}dt_{j+1}dt_ndt_{j}.	\end{align*}
	Hence, denoting \begin{align*}
	&E^{\boldsymbol{\rho}}_n :=\\
	&\int_0^T \int_{w}^T \overline{\bZeta\boldsymbol{\rho}_\ell}^{i,.}(v) \bPE_{n-1}(v-w)\bM(t_1) dv dw\\
	&+ \sum_{j=1}^{n-1}\int_{0}^{T}\int_{t_j}^T\bZetaE^{i,.}(T-t_n)\int_{t_j}^{t_{n}}\bPE_{n-j-1}(t_n-t_{j+1})\overline{\bP\boldsymbol{\rho}_\ell}(t_{j+1}-t_{j}) \int_{0}^{t_{j}}\bPE_{j-1}(t_j-t_{1})\bM(t_1)dt_{1}dt_{j+1}dt_ndt_{j}
		\end{align*}
		$$ A^{\boldsymbol{\rho}} = \sum_{n=1}^{+\infty}E^{\boldsymbol{\rho}}_n.$$	
	\underline{Step 5 :} Summing in $n$. \\
	$\bullet$  If $\sum \boldsymbol{\rho}_\ell \equiv 1$, using that $\bPE_0$ is the Dirac  distribution in 0,
	\begin{align*}
		A^{\not\boldsymbol{\rho}} &=\sum_{n=1}^{+\infty}\int_0^T \int_{w}^T\bZetaE^{i,.}(T-v)\bPE_{n-1}(v - w)\boldsymbol{\bM}(w)dv dw\\
		&=\int_0^T \int_{w}^T\bZetaE^{i,.}(T-v)\left(\bPs(v - w) + \bPE_0(v - w)\right)\boldsymbol{\bM}(w)dv dw\\
		&=\int_0^T \bZetaE^{i,.}(T-v)\left(\boldsymbol{\bM}(v)+ \int_{0}^v\bPs(v - w)\boldsymbol{\bM}(w)dw\right) dv = \E[\boldsymbol{Z}^i_T].
	\end{align*}
	$\bullet$  If $\sum \boldsymbol{\rho}_\ell \not\equiv 1$,  the first term in $E^{\boldsymbol{\rho}}_n$ sums as above
$$\sum_{n=1}^{+\infty}\int_0^T \int_w^T \overline{\bZeta\boldsymbol{\rho}_\ell}^{i,.}(v) \bPE_{n-1}(v-w)\bM(w) dv dw=	
	\E\left[(\boldsymbol{Z}^{\boldsymbol{\bZeta\rho}_\ell}_T)^i\right] $$
	and the second term sums as, denoting $\boldsymbol{Z^{\bZeta \rho_\ell}}$ a $(\bZeta \boldsymbol{\rho}_\ell,\bM,\bP)-$MSPD
{\small	\begin{align*}
	&\hspace*{-0.6cm}\sum_{n=1}^{+\infty}\sum_{j=1}^{n-1}\int_{0}^{T}\int_{t_j}^T\bZetaE^{i,.}(T-t_n)\int_{t_j}^{t_{n}}\bPE_{n-j-1}(t_n-t_{j+1})\overline{\bP\boldsymbol{\rho}_\ell}(t_{j+1}-t_{j})
		\int_{0}^{t_{j}}\bPE_{j-1}(t_j-t_{1})\bM(t_1)dt_{1}dt_{j+1}dt_ndt_{j}\\
		&\hspace*{-0.6cm}=\sum_{j=1}^{+\infty}\sum_{n=j+1}^{+\infty}\int_{0}^{T}\int_{t_j}^T\bZetaE^{i,.}(T-t_n)\int_{t_j}^{t_{n}}\bPE_{n-j-1}(t_n-t_{j+1})\overline{\bP\boldsymbol{\rho}_\ell}(t_{j+1}-t_{j})
		\int_{0}^{t_{j}}\bPE_{j-1}(t_j-t_{1})\bM(t_1)dt_{1}dt_{j+1}dt_ndt_{j}\\
		&\hspace*{-0.6cm} =\int_{0}^{T} \hspace*{-2mm}  \int_{t_j}^T \hspace*{-2mm}  \bZetaE^{i,.}(T-t_n)\int_{t_j}^{t_{n}}\hspace*{-3mm}  \left(\bPs(t_n-t_{j+1}) + \bPE_{0}(t_n-t_{j+1}) \right)\overline{\bP\boldsymbol{\rho}_\ell}(t_{j+1},t_{j})
		\int_{0}^{t_j} \hspace*{-2mm} \left(\bPs(t_j-t_{1}) + \bPE_{0}(t_j-t_{1}) \right)\bM(t_1)dt_{1}dt_{j+1}dt_ndt_{j}\\
		&\hspace*{-0.6cm} =\int_{0}^T \hspace*{-2mm} \bZetaE^{i,.}(T-t_n)\int_{0}^{t_{n}}\hspace*{-2mm} \left(\bPs(t_n-t_{j+1}) + \bPE_{0}(t_n-t_{j+1}) \right)
		\int_{0}^{t_{j+1}}\hspace*{-4mm} \overline{\bP\boldsymbol{\rho}_\ell}(t_{j+1}-t_{j})\left(\int_{0}^{t_j}\bPs(t_j-t_{1})\bM(t_1)dt_1 + \bM(t_j) \right)dt_{j}dt_{j+1}dt_n\\
		&\hspace*{-0.6cm} =\int_{0}^T\bZetaE^{i,.}(T-t_n)\int_{0}^{t_{n}}\hspace*{-2mm}\left(\bPs(t_n-t_{j+1}) + \bPE_{0}(t_n-t_{j+1}) \right)\E[\boldsymbol{Z_{t_{j+1}}^{\bP\rho_\ell}}]dt_{j+1}dt_n\\
		&\hspace*{-0.6cm}=\int_{0}^T\bZetaE^{i,.}(T-v)\left(\E[\boldsymbol{Z_{v}^{\bP\rho_\ell}}] + \int_{0}^{v}\bPs(v-w)\E[\boldsymbol{Z_{w}^{\bP\rho_\ell}}]dw\right)dv
	\end{align*}
	}
which concludes the proof.	
\end{proof}

\paragraph{Acknowledgment} This work is part of the project CyFi "Cyber Financialization" with the financial support of Bpifrance. Caroline Hillairet's research benefited from the support of the Chair Stress Test, RISK Management and Financial Steering, led by the French Ecole Polytechnique and its Foundation and sponsored by BNP Paribas.

\bibliographystyle{plain}
\bibliography{biblioHR}
\end{document}